\documentclass[a4paper,12pt]{amsart}

\usepackage{amssymb}
\usepackage{latexsym}
\usepackage{amsmath}
\usepackage{amscd}
\usepackage{graphicx}

\title[Galois obstruction]
{The first Galois obstruction in the Johnson cokernel}

\author{Shigeyuki Morita}
\address{Graduate School of Mathematical Sciences, 
The University of Tokyo, 
3-8-1 Komaba, 
Meguro-ku, Tokyo, 153-8914, Japan}
\email{shmori314@gmail.com}
\author{Takuya Sakasai}
\address{Graduate School of Mathematical Sciences, 
The University of Tokyo, 
3-8-1 Komaba, 
Meguro-ku, Tokyo, 153-8914, Japan}
\email{sakasai@ms.u-tokyo.ac.jp}
\author{Masaaki Suzuki}
\address{Department of Frontier Media Science, 
Meiji University, 
4-21-1 Nakano, Nakano-ku, Tokyo, 164-8525, Japan}
\email{mackysuzuki@meiji.ac.jp}

\thanks{
The authors were partially supported by KAKENHI (No.24K06740, 26H01994, No.26K22247), 
Japan Society for the Promotion of Science, Japan.}

\dedicatory{Dedicated to Hiroaki Nakamura in celebration of his 60th birthday}


\newtheorem{thm}{Theorem}[section]
\newtheorem{prop}[thm]{Proposition}
\newtheorem{lem}[thm]{Lemma}
\newtheorem{cor}[thm]{Corollary}
\theoremstyle{definition}
\newtheorem{definition}[thm]{Definition}

\newtheorem{remark}[thm]{Remark}

\newtheorem{question}[thm]{Question}

\begin{document}

\newcommand{\Ker}{\mathop{\mathrm{Ker}}\nolimits}
\newcommand{\Hom}{\mathop{\mathrm{Hom}}\nolimits}
\renewcommand{\Im}{\mathop{\mathrm{Im}}\nolimits}
\newcommand{\Coker}{\mathop{\mathrm{Coker}}\nolimits}

\newcommand{\Der}{\mathop{\mathrm{Der}}\nolimits}
\newcommand{\Out}{\mathop{\mathrm{Out}}\nolimits}
\newcommand{\Aut}{\mathop{\mathrm{Aut}}\nolimits}
\newcommand{\Q}{\mathbb{Q}}
\newcommand{\Z}{\mathbb{Z}}
\newcommand{\R}{\mathbb{R}}

\begin{abstract}
We explicitly determine the first Galois obstruction in the cokernel of the Johnson homomorphism 
of the mapping class group of a surface of genus $g$, for every genus $g \ge 2$.
It is described as a sum of two terms which are considerably different in character.
One lies in the kernel of the Enomoto-Satoh trace map, whereas
the other belongs to a certain ideal which vanishes 
upon passage to the closed surface case. 
\end{abstract}

\renewcommand\baselinestretch{1.1}
\setlength{\baselineskip}{16pt}

\newcounter{fig}
\setcounter{fig}{0}

\maketitle


\section{Introduction}
Let $\Sigma_{g,1}$ be a compact oriented surface of genus $g$ with one boundary
component and let $\mathcal{M}_{g,1}$ be the mapping class group of
$\Sigma_{g,1}$ relative to the boundary. 
The group $\mathcal{M}_{g,1}$ naturally acts on the fundamental group $\pi_1 (\Sigma_{g,1})$ of $\Sigma_{g,1}$, 
which is isomorphic to the free group of rank $2g$.  
The Dehn-Nielsen-Baer theorem implies that the induced homomorphism 
$\mathcal{M}_{g,1} \to \Aut \pi_1 (\Sigma_{g,1})$ is injective. 
We denote by 
$\{\mathcal{M}_{g,1}(k); k=0, 1,2,\ldots\}$ the {\it Johnson filtration} of $\mathcal{M}_{g,1}$. 
Here $\mathcal{M}_{g,1}(k)$ is the subgroup of $\mathcal{M}_{g,1}$ consisting of 
mapping classes which trivially act on the $k$-th nilpotent quotient of $\pi_1 (\Sigma_{g,1})$. 
In this notation, $\mathcal{M}_{g,1}(0)=\mathcal{M}_{g,1}$, and 
$\mathcal{M}_{g,1} (1)$ is the subgroup known as the Torelli group, which acts trivially on 
the abelianization $H_\Z:=H_1 (\Sigma_{g,1};\mathbb{Z})$ of $\pi_1 (\Sigma_{g,1})$. 
Generalizing Johnson's pioneering work \cite{j,j2}, the first author \cite{morita93} defined  
the $k$-th Johnson homomorphism
\[\tau_{g,1}^\Z (k) \colon \mathcal{M}_{g,1}(k) \longrightarrow \mathfrak{h}_{g,1}^\Z (k)\]
to investigate the structure of $\mathcal{M}_{g,1}(k)$. Here  
\[\mathfrak{h}_{g,1}^\Z=\bigoplus_{k=1}^{\infty} \mathfrak{h}_{g,1}^\Z (k)\]
denotes the graded Lie algebra consisting of positive-degree symplectic derivations of the
free Lie algebra $\mathcal{L}_{g,1}^\Z=\oplus_{k=1}^\infty \mathcal{L}_{g,1}^\Z (k)$ 
generated by $H_\Z$. 
The totality of these homomorphisms for all $k\geq 1$ induces an
embedding of graded Lie algebras
\[\tau_{g,1}^\Z \colon \bigoplus_{k=1}^\infty \left(\mathcal{M}_{g,1}(k)/\mathcal{M}_{g,1}(k+1)\right) 
\hookrightarrow \mathfrak{h}_{g,1}^\Z.\]
It is known that this embedding is $\left(\mathcal{M}_{g,1}/\mathcal{M}_{g,1}(1)\right)$-equivariant.  
Here $\mathcal{M}_{g,1}/\mathcal{M}_{g,1}(1)$ is naturally identified with 
the symplectic automorphism group $\mathrm{Sp} (H_\Z)$ of 
$H_\Z$ with respect to the intersection form
\[\mu \colon H_\Z \otimes H_\Z \longrightarrow \Z \qquad (x \otimes y \longmapsto x \cdot y),\]
which is non-degenerate and skew-symmetric. 
Explicitly, $\mathfrak{h}_{g,1}^\Z (k)$ is given by 
\[\mathfrak{h}_{g,1}^\Z (k)=\Ker \left ( H_\Z \otimes \mathcal{L}_{g,1}^\Z (k+1) \xrightarrow{[ \cdot , \cdot]} 
\mathcal{L}_{g,1}^\Z (k+2)\right)\]
and it is naturally embedded in $H_\Z^{\otimes (k+2)}$ as an $\mathrm{Sp} (H_\Z)$-module. 

Now the following two problems are fundamental in the theory of Johnson homomorphisms:
\begin{itemize}
\item[(1)] Determine the image of $\tau_{g,1}^\Z$.
\item[(2)] Give a {\it topological} meaning of the cokernel $\Coker \tau_{g,1}^\Z$. 
\end{itemize}

As for problem (1), building on results of 
Johnson \cite{j,j2}, the first author \cite{morita93, morita99}, Hain \cite{haint} and 
Asada-Nakamura \cite{an}, 
the authors computed the rational image up to degree $7$. 
More recently, works by Kupers and Randal-Williams \cite{KRW}, 
Felder-Naef-Willwacher \cite{FNW} and Naef-Willwacher \cite{NW} have provided 
a general description of the rational image. 

In the authors' computation of the images of Johnson homomorphisms, they made heavy use of 
the {\it Enomoto-Satoh trace map} introduced in \cite{es}. 
For each degree $k \ge 3$, Enomoto and Satoh defined a homomorphism
\[ES_k^\Z: \mathfrak{h}_{g,1}^\Z (k)\longrightarrow \mathfrak{a}_g^\Z (k-2)=(H_\Z^{\otimes k})_{\Z/k \Z},\]
where the target denotes the degree $(k-2)$ part of the associative version among  
Kontsevich's symplectic derivation Lie algebras (over $\Z$), and 
$(H_\Z^{\otimes k})_{\Z/k \Z}$ denotes the coinvariants of $H_\Z^{\otimes k}$ 
under the cyclic action on the tensor factors. 
They proved the important result that
$ES_k^\Z$ vanishes identically on the Johnson image $\Im \tau_{g,1}^\Z (k)\subset \mathfrak{h}_{g,1}^\Z (k)$.
Therefore, we may use $ES_k^\Z$ to detect the cokernel of $\tau_{g,1}^\Z (k)$. 

As for the problem (2), there have been several (partial) answers. 
One approach comes from the theory of invariants of three-dimensional manifolds, 
particularly from the theory of homology cobordisms over surfaces.
Here we refer only to papers of Garoufalidis-Levine \cite{gl} and Habiro-Massuyeau \cite{hm} 
for their close connections with Johnson homomorphisms. 
Another approach, discovered earlier, comes from number theory. 

In what follows, we consider the rational versions of the above objects. 
That is, we tensor the above modules with the rationals $\Q$ 
and consider the homomorphism
\[\tau_{g,1}:=\tau_{g,1}^\Z \otimes \Q \colon 
\bigoplus_{k=1}^\infty \left(\mathcal{M}_{g,1}(k)/\mathcal{M}_{g,1}(k+1)\right) \otimes \Q 
\hookrightarrow \mathfrak{h}_{g,1}.\]
Here $\mathfrak{h}_{g,1}:=\mathfrak{h}_{g,1}^\Z \otimes \Q$. 
We also use analogous notation: 
$H:=H_\Z \otimes \Q$, $\mathcal{L}_{g,1}=\mathcal{L}_{g,1}^\Z \otimes \Q$, 
$ES_k:=ES_k^\Z \otimes \Q$ and $\mathfrak{a}_g (k):=\mathfrak{a}_g^\Z (k) \otimes \Q$. 
We often abbreviate $\mathrm{Sp} (H)$ to $\mathrm{Sp}$. 

The (graded version of) {\it arithmetic} Johnson homomorphism introduced in 
Hain \cite{hain} gives rise to an injective homomorphism 
\[\hat{\tau}_{g,1} \colon 
\mathcal{L}(\sigma_3,\sigma_5,\ldots)\longrightarrow
(\Coker \tau_{g,1})^{\mathrm{Sp}}= (\mathfrak{h}_{g,1}/\Im \tau_{g,1})^{\mathrm{Sp}}\]
where the left-hand side denotes the motivic Lie algebra, which is a free Lie algebra generated by 
elements $\sigma_{2k+1}\ (k=1,2,\ldots)$ of degree $4k+2$,
and the right-hand side denotes the $\mathrm{Sp}$-invariant part of
the quotient of $\mathfrak{h}_{g,1}$
divided by the image $\Im \tau_{g,1}$ of the Johnson homomorphism.
Hain explained the authors that the above homomorphism is not canonically defined.
It depends on choice of free generators for its image 
$\mathrm{Im}\, \hat{\sigma}_{g,1}$ (see Hain's paper \cite{hain}, especially Remark 9.1). 
Elements of the image of the above homomorphism $\hat{\tau}_{g,1}$ are
called {\it Galois obstructions}, whose representatives in $\mathfrak h_{g,1}$ are well-defined 
up to a nonzero scalar and modulo the Johnson image $\Im \tau_{g,1}$. 
Along the lines of deep theories of Grothendieck, Ihara and Deligne,
the appearance of these elements was conjectured by
Oda already in the 1980s. 
Their existence was first established independently by Nakamura \cite{nakamura} 
and Matsumoto \cite{matsumoto}, 
and the whole of Oda's conjecture was proved by Takao \cite{takao}. 
The injectivity of the above homomorphism ultimately follows from a fundamental result of Brown \cite{brown}. 
The problem is then to determine how these Galois obstructions sit inside $\mathfrak{h}_{g,1}$.
We refer to Hain's article cited above as well as Matsumoto's article
\cite{matsumotopcmi}
for details.


The purpose of this paper is to study the relationship between the Enomoto-Satoh trace map 
and Galois obstructions in degree 6, the first degree in which the two objects meet.
We also describe the first Galois obstruction $\hat{\tau}_{g,\ast}(6)(\sigma_3)$ explicitly. 
Our main results are stated in the next section.

Our computations are carried out mainly in terms of $\mathrm{Sp}$-modules. 
We refer to Fulton-Harris \cite{fh} for details of the representation theory of symplectic groups. 
Here we label irreducible polynomial representations of $\mathrm{Sp} (H)$ by Young diagrams.  
We will use the following four irreducible representations, all of which already appeared in Johnson's work \cite{j}: 
\begin{align*}
&[0] = \mathbb{Q} \ \ \text{(trivial $\mathrm{Sp}(H)$-module)}, \qquad [1] = H, \qquad
[1^2]=\wedge^2 H/\mathbb{Q}, \\
&[1^3]= (\wedge^3 H)/H, \qquad [2] =S^2 H = \mathfrak{sp} (H) \ \ 
\text{(the Lie algebra of $\mathrm{Sp}(H)$)}.
\end{align*}
\noindent
The {\it symplectic element} 
$\omega_0 \in (\wedge^2 H)^\mathrm{Sp} = (H^{\otimes 2})^\mathrm{Sp} \cong \mathbb{Q}$ 
is defined by \begin{align*}
\omega_0&=x_1\wedge y_1+\cdots + x_g\wedge y_g \ \ \text{(in $\wedge^2 H$)}\\
&= x_1\otimes y_1-y_1\otimes x_1+\cdots + x_g\otimes y_g-y_g\otimes x_g \ \ \text{(in $H^{\otimes 2}$)}
\end{align*}
using a symplectic basis $\{x_1, y_1, \ldots, x_g, y_g\}$ of $H$. 
We have the $\mathrm{Sp}$-irreducible decomposition $\wedge^2 H=[1^2] \oplus [0]$, 
where $[0]$ is generated by $\omega_0$.


\bigskip
\noindent
{\it Acknowledgements.} \ 
This work was motivated mainly by Hain's article \cite{hain},
in particular the discussion in \S 1.9 (``the most optimistic landscape'')
and Question 12.8.
The authors would like to thank Richard Hain 
for enlightening communications as well as helpful comments
to an earlier draft of this paper. Also, the first author would like to express 
his hearty thanks to 
Nariya Kawazumi, Makoto Matsumoto and Hiroaki Nakamura 
for helpful information and many valuable discussions 
over many years.

\section{Statement of the main results}\label{sec:results}
In our paper \cite[Section 7]{mss4}, we proved that 
the restriction of the Enomoto-Satoh trace map to the $\mathrm{Sp}$-invariant part
\[ES_6 \colon \mathfrak{h}_{g,1}(6)^{\mathrm{Sp}}\cong
\begin{cases}
\Q^5 \ (g\geq 3)\\
\Q^4 \ (g=2)
\end{cases}
\longrightarrow \quad  
\mathfrak{a}_g(4)^{\mathrm{Sp}}\cong\Q^2 \]
is surjective, and hence
\[\Ker ES_6\cong
\begin{cases}
\Q^3 \ (g\geq 3)\\
\Q^2 \ (g=2)
\end{cases}
\ \text{while}\quad
\Im \tau_{g,1}(6)^{\mathrm{Sp}}\cong
\begin{cases}
\Q^2\ (g\geq 3)\\
\Q \ \ (g=2)
\end{cases}.\]
Therefore we have a short exact sequence 
\begin{equation}
0\longrightarrow \Im \tau_{g,1}(6)^{\mathrm{Sp}}
\longrightarrow \Ker ES_6\longrightarrow \Q
\longrightarrow 0
\label{eq:es}
\end{equation}
\noindent 
for all $g\geq 2$. Thus, there exists an extra element,
denoted by $\varepsilon$, in $\Ker ES_6$ that does not belong to the Johnson image.

Now let  $\mathcal{M}_{g,\ast}$ be the mapping class group of an oriented closed surface
$\Sigma_g$ of genus $g$ fixing a base point $\ast \in\Sigma_g$. 
We have the Johnson filtration $\{\mathcal{M}_{g,\ast }(k); k=1,2,\ldots\}$ 
and the $k$-th Johnson homomorphism
\[\tau_{g,\ast}^\Z (k) \colon \mathcal{M}_{g,\ast}(k)\longrightarrow\mathfrak{h}_{g,\ast}^\Z(k),\]
where $\mathfrak{h}_{g,\ast}^\Z=\oplus_{k=1}^{\infty} \mathfrak{h}_{g,\ast}^\Z (k)$ is also 
a graded Lie algebra. We have a natural projection 
$\mathfrak{h}_{g,1}^\Z \twoheadrightarrow \mathfrak{h}_{g,\ast}^\Z$,  
and we denote the kernel by $\mathfrak{j}_{g,1}^\Z$, 
a Lie ideal of $\mathfrak{h}_{g,1}^\Z$ (see \cite[Section 3]{mss4}). 
We denote their rational versions by $\tau_{g,\ast}$, $\mathfrak{h}_{g,\ast}$ and $\mathfrak{j}_{g,1}$. 
We have $\mathfrak{h}_{g,\ast}(k)=\mathfrak{h}_{g,1}(k)/\mathfrak{j}_{g,1}(k)$. 
It was proved in \cite[Proposition 3.5]{mss4} that $\Im \tau_{g,1}^\Z (k) 
\cap \mathfrak{j}_{g,1}^\Z (k)=\{0\}$ for all $k \ge 3$. 
In Section 7 of the same paper, 
we also proved that the restriction of $ES_6$ to $\mathfrak{j}_{g,1}(6)^{\mathrm{Sp}}$ 
gives an isomorphism 
$ES_6 \colon  \mathfrak{j}_{g,1}(6)^{\mathrm{Sp}} \xrightarrow{\cong} \Q^2$. 
Hence we have a direct sum decomposition
\begin{equation}
\mathfrak{h}_{g,1}(6)^{\mathrm{Sp}}=\Ker ES_6\oplus 
\mathfrak{j}_{g,1}(6)^{\mathrm{Sp}}.
\label{eq:ds}
\end{equation}

It is known that the (graded version of) arithmetic Johnson homomorphism 
descends to the pointed-surface setting associated with $\tau_{g,\ast}$,  
and hence induces injective homomorphism 
\[\hat{\tau}_{g,\ast} \colon
\mathcal{L}(\sigma_3,\sigma_5,\ldots)\longrightarrow
(\Coker \tau_{g,*})^{\mathrm{Sp}}=
(\mathfrak{h}_{g,*}/\Im \tau_{g,*})^{\mathrm{Sp}}.\]
The first Galois obstruction appears in degree $6$. Therefore 
\[\hat{\tau}_{g,\ast}(6) (\sigma_3)\in \mathfrak{h}_{g,*}(6)^{\mathrm{Sp}}=\mathfrak{h}_{g,1}(6)^{\mathrm{Sp}}
/\mathfrak{j}_{g,1}(6)^{\mathrm{Sp}}\cong \Ker ES_6.\]
We can now conclude that, in the above direct sum decomposition \eqref{eq:ds},
the $\Ker ES_6$-component of the first Galois obstruction
\[\sigma_3(g)=\hat{\tau}_{g,1}(6)(\sigma_3)\in \mathfrak{h}_{g,1}(6)^{\mathrm{Sp}}
= \Ker ES_6\oplus \mathfrak{j}_{g,1}(6)^{\mathrm{Sp}}\]
is non-trivial. 
Indeed, this part must be equal to a nonzero scalar multiple of 
the extra element $\varepsilon$, since it is uniquely defined up to a nonzero scalar and 
modulo the Johnson image (as mentioned in Introduction).

In \cite[Section 7]{mss4}, we constructed a basis $v_1, v_2, v_3, v_4, v_5$ of 
$\mathfrak{h}_{g,1}(6)^{\mathrm{Sp}}\cong\Q^5$ for $g\geq 3$. For $g=2$, 
we have $v_5=0$, and $v_1, v_2, v_3, v_4$ form a basis of $\mathfrak{h}_{2,1}(6)^{\mathrm{Sp}}\cong\Q^4$. 
The following is the main result of the present paper.

\begin{thm}
The first Galois obstruction 
$$
\sigma_3(g)=\hat{\tau}_{g,1}(6)(\sigma_3)\in \mathfrak{h}_{g,1}(6)^{\mathrm{Sp}}
=\Ker ES_6\oplus \mathfrak{j}_{g,1}(6)^{\mathrm{Sp}}
$$
in the cokernel of the Johnson homomorphism is described as a sum 
of two terms as follows.
$$
\sigma_3(g)=180(g-1)g \ \varepsilon-j
$$ where
\begin{align*}
\varepsilon=& \frac{2}{2g(2g+1)(2g+2)(2g+3)}v_1
+\frac{8}{(2g-2)2g(2g+1)(2g+2)}v_3\\
& +\frac{3}{(2g-2)(2g-1)2g(2g+1)}v_4\quad \in \Ker ES_6\\
j&=-9v_1 -7v_2-19v_3 -5v_4+10v_5\ \in \mathfrak{j}_{g,1}(6)^{\mathrm{Sp}}.
\end{align*}
The element $\varepsilon$ is determined up to a nonzero scalar and modulo the Johnson image
$\Im \tau_{g,1}(6)^{\mathrm{Sp}}\subset \Ker ES_6$
while $j$ is determined up to a nonzero scalar.
\label{th:main}
\end{thm}

\begin{remark}\label{rem:inv_stable}
Let $H_{(g)}:=H_1 (\Sigma_{g,1};\Q)$. The natural inclusion $H_{(g)} \hookrightarrow H_{(g+1)}$ induced by 
the standard embedding $\Sigma_{g,1} \hookrightarrow \Sigma_{g+1,1}$ gives an inclusion 
$H_{(g)}^{\otimes (2k+2)} \hookrightarrow H_{(g+1)}^{\otimes (2k+2)}$.
However, it does {\it not} map $(H_{(g)}^{\otimes (2k+2)})^{\mathrm{Sp}(H_{(g)})}$ into 
$(H_{(g+1)}^{\otimes (2k+2)})^{\mathrm{Sp}(H_{(g+1)})}$. 
The same is true for the natural inclusion $\mathfrak{h}_{g,1}(2k) \hookrightarrow \mathfrak{h}_{g+1,1}(2k)$. 
On the other hand, it is classically known that the invariant tensors in $H_{(g)}^{\otimes (2k+2)}$ 
are described by linear chord diagrams with $2k+2$ vertices. 
More precisely, the tensors associated with linear chord diagrams form 
a basis of $(H_{(g)}^{\otimes (2k+2)})^{\mathrm{Sp}(H_{(g)})}$ when $g \ge k+1$. 
Consequently, $\mathfrak{h}_{g,1}(2k)^{\mathrm{Sp} (H_{(g)})}$ also 
admits a description in terms of linear chord diagrams 
(see \cite[Section 4]{morita99} and \cite[Sections 2, 6 and 7]{mss4}). 
Throughout this paper, we identity 
\[\mathfrak{h}_{3,1}(6))^{\mathrm{Sp} (H_{(3)})} \cong \mathfrak{h}_{4,1}(6)^{\mathrm{Sp} (H_{(4)})} 
\cong \cdots \cong \mathfrak{h}_{g,1}(6)^{\mathrm{Sp} (H_{(g)})} \cong \cdots\]
by identifying basis elements $v_1,v_2,v_3,v_4,v_5$ defined by the same chord-diagram formulas.
For $g=2$, we have $v_5=0$, and $v_1,v_2,v_3,v_4$ form a basis of $\mathfrak{h}_{2,1}(6)^{\mathrm{Sp}(H_{(2)})}$. 
Thus we identify $\mathfrak{h}_{2,1}(6)^{\mathrm{Sp}(H_{(2)})}$ with the subspace of 
$\mathfrak{h}_{3,1}(6)^{\mathrm{Sp}(H_{(3)})}$ spanned by $v_1,v_2,v_3,v_4$. 
\end{remark}

For some time, it had been an interesting question whether the Enomoto-Satoh trace map vanishes on the Galois obstructions or not. This problem was recently settled in Hain's article \cite{hain}
as follows. Kawazumi and Kuno \cite{kk} developed their theory 
of the Goldman-Turaev Lie bialgebra in which they gave a geometric description of
the Johnson homomorphism by replacing $\mathfrak{h}_{g,1}$ with this 
Lie bialgebra. Here the Turaev cobracket appears as an obstruction for the 
Johnson image. Alekseev, Kawazumi, Kuno and Naef \cite{akkn} proved that 
the kernel of this obstruction coincides with that of the Enomoto-Satoh obstruction. 
In the above article, Hain gives a survey of their theory 
together with his own results from the viewpoint of Hodge theory. In particular, 
Theorem 12.6 of that paper shows that the Turaev cobracket of 
$\sigma_{2k+1}(g)$ is non-trivial for all $k$. 
Hain notes that this result is essentially due to
Alekseev, Kawazumi, Kuno and Naef.
Since the kernel of the Turaev cobracket coincide with that of the Enomoto-Satoh trace map, 
as mentioned above, 
it follows that $ES_{4k+2}(\sigma_{2k+1}(g))\not=0$ for all $k \ge 1$ and $g \ge 2$.

We now determine the value of the Enomoto-Satoh trace map on the 
first Galois obstruction as follows. 
The restriction of the Enomoto-Satoh trace map $ES_6$ 
to the $\mathrm{Sp}$-invariant part decomposes as
\[ES_6=ES_6^1\oplus ES_6^2 \colon \mathfrak{h}_{g,1} (6)^\mathrm{Sp} \longrightarrow 
\mathfrak{a}_{g} (4)^\mathrm{Sp} \cong \mathbb{Q}^2.\]
Here
\[ES_6^i=
\begin{cases}
(12)(34)(56)\circ ES_6 & (i=1),\\
(12)(35)(46)\circ ES_6 & (i=2)
\end{cases}\]
and
\[(12)(34)(56), (12)(35)(46) \colon \mathfrak{a}_g(4)^{\mathrm{Sp}}\longrightarrow \Q\]
are the multi-contractions corresponding to $(12)(34)(56)$ and $(12)(35)(46)$, respectively. 
See \cite[Section 4.1]{morita99} for details on multi-contractions.

As already mentioned, the first Galois obstruction 
$\sigma_3(g)$ is well-defined up to a nonzero scalar and modulo the Johnson image 
and the Enomoto-Satoh trace map vanishes on the Johnson image.
It follows that the ratio $ES_6^1(\sigma_3(g))/ES_6^2(\sigma_3(g)) \in \mathbb{Q} \cup \{\infty\}$ 
is well-defined and completely determines the projective class of $ES_6(\sigma_3(g))$.

\begin{thm}
The value of the Enomoto-Satoh trace map
on the  first Galois obstruction $\sigma_3(g)$ for $g \geq 2$ is given by
\[ES_6^1(\sigma_3(g))/ES_6^2(\sigma_3(g))=2g-1.\]
\label{th:es}

\end{thm}
Now we briefly mention our method of proving the main result. 
Let $\mathfrak{sp}(H)$ denote the Lie algebra of $\mathrm{Sp}(H)$. 
It acts on $\mathfrak{h}_{g,1}$, and the semidirect product 
$\mathfrak{h}_{g,1}\rtimes \mathfrak{sp}(H)$ becomes a graded Lie algebra 
with $\mathfrak{sp}(H)$ as its degree $0$ part. 
The {\it normalizer} of the (extended) Johnson image $\Im \tau_{g,1}\rtimes \mathfrak{sp}(H)$ in 
$\mathfrak{h}_{g,1}\rtimes \mathfrak{sp}(H)$ is defined as 
$\mathcal{N}=\oplus_{k=1}^\infty \mathcal{N}(k)$ with 
\[\mathcal{N}(k)=\{\varphi\in \mathfrak{h}_{g,1}(k) \mid \text{$[\varphi,\psi]\in \Im \tau_{g,1}(k+l)$
for any $l=0,1,\ldots$ and $\psi\in \Im \tau_{g,1}(l)$}\},\]
where we set $\Im \tau_{g,1}(0)=\mathfrak{sp}(H)$.
Following Hain \cite{hain},
here we change the definition of the normalizer $\mathcal{N}$ given in 
our paper \cite{mss4}
by including a condition coming from the degree $0$ part $\mathfrak{sp}(H)$ of 
$\mathfrak{h}_{g,1}\rtimes \mathfrak{sp}(H)$. 
Then the quotient
$\mathcal{N}/ \Im\tau_{g,1}$ is a trivial $\mathrm{Sp}$-module.
This is because 
$\mathfrak{sp}(H) = [2]$ and 
\[[\mathfrak{sp}(H),\lambda]=
\begin{cases}
\lambda & (\lambda\not= [0])\\
\{0\} & (\lambda=[0]).
\end{cases}\]
for any irreducible summand $\lambda\subset \mathfrak{h}_{g,1}(k)$. 
This should be important in the study of the arithmetic mapping class group because
it is known that 
the Galois obstructions appear in $\mathfrak{h}_{g,1}^{\mathrm{Sp}}$
and they normalize the image of the Johnson homomorphism, 
that is, they are contained in $\mathcal{N}$.

We prove Theorem \ref{th:main} by using this property of $\mathcal{N}$. More 
precisely, we proved in \cite[Theorem 7.5]{mss4} that there exists a short exact sequence
\[0\longrightarrow \Im\tau_{g,1}(6)\longrightarrow \mathcal{N}(6)
\longrightarrow \Q \longrightarrow 0\]
for all $g\geq 3$, and it also holds for the case $g=2$ as is shown in this paper.
It follows that $\mathcal{N}(6)/\Im\tau_{g,1}(6) \cong \Q$ is generated by the first Galois 
obstruction $\hat{\tau}_{g,1}(6)(\sigma_3)$. The main ingredient of this paper
is the explicit  determination of $\mathcal{N}(6)$.

We conclude this section with an outline of the proof. 
In 
$\S$\ref{in}, keeping in mind the short exact sequence \eqref{eq:es}, 
we add an element $\varepsilon$, called the {\it extra element}, 
to a basis of $\Im\tau_{g,1}(6)^{\mathrm{Sp}}$ given in \cite{mss4} 
to obtain a basis of $\Ker ES_6$ (see Theorem \ref{th:extra}).

Next, in order to measure the failure of this extra element $\varepsilon$ 
to normalize the Johnson image, we analyze the $[1^3]$- and $[1]$-isotypical components 
of $\mathfrak{h}_{g,1}(7)$ in 
$\S$\ref{sec:1_13}. We show that
\[ [\Im \tau_{g,1}(1), \varepsilon]\not\subset \Im\tau_{g,1}(7).\]  
In 
$\S$\ref{sec:polyn}, we prove that the values of the various homomorphisms 
which we use in 
$\S$\ref{sec:1_13} are polynomials in $g$.
Then in 
$\S$\ref{sec:firstGal}, we construct an element 
$j\in \mathfrak{j}_{g,1}(6)^{\mathrm{Sp}}$ in terms of the basis 
\begin{align*}
j_1&=v_2-5v_3-4v_4+2v_5\\
j_2&=3v_1+3v_2+3v_3-v_4-2v_5
\end{align*}
of $\mathfrak{j}_{g,1}(6)^{\mathrm{Sp}}$ given in \cite{mss4},
such that 
a suitable linear combination of $\varepsilon$ and $j$
does normalize the Johnson image.
Therefore we can conclude that this element 
represents the first Galois obstruction $\sigma_3(g)$.


\section{Kernel of the Enomoto-Satoh homomorphism in degree $6$}\label{in}

As mentioned in Section \ref{sec:results}, 
$\dim \Ker  ES_6$ is equal to $3$ for $g \ge 3$ and to $2$ for $g=2$. 
We now construct a basis of $\Ker  ES_6$ in terms of the normalized basis 
$\bar{v}_i\ (i=1,2,3,4,5)$ of $\mathfrak{h}_{g,1}(6)^{\mathrm{Sp}}$ introduced in \cite{mss4}.

\begin{thm}
The following elements form a basis of $\Ker  ES_6$.

\noindent
$\mathrm{(i)}\ $ The case of $g\geq 3$
\begin{align*}
&\tau_1=3 \bar{v}_1-\bar{v}_2+8\bar{v}_3+5\bar{v}_5,\quad
\tau_2=6\bar{v}_1+3\bar{v}_2+36\bar{v}_3+25\bar{v}_4-25\bar{v}_5
\in \Im\tau_{g,1}(6)^{\mathrm{Sp}} \\ 
&\text{and} \ \ 2\bar{v}_1+8\bar{v}_3+3\bar{v}_4
\end{align*}
$\mathrm{(ii)}\ $ The case of $g=2$
$$
5\tau_1+\tau_2= 21\bar{v}_1-2\bar{v}_2+76\bar{v}_3+25\bar{v}_4
\in \Im\tau_{2,1}(6)^{\mathrm{Sp}} \quad\text{and}\quad
2\bar{v}_1+8\bar{v}_3+3\bar{v}_4
$$
\label{th:extra}
\end{thm}
\begin{proof}
Table \ref{tab:ES} indicates the values of the Enomoto-Satoh trace maps
$ES_6=\left(\begin{smallmatrix}ES_6^1\\ES_6^2\end{smallmatrix}\right)$ 
on each element $v_i\ (i=1,2,3,4,5)$ of the basis
for genus $1\leq g\leq 7$.
Dividing these values by the corresponding eigenvalues given in Table \ref{tab:ev},
we obtain the values on each element $\bar{v}_i\ (i=1,2,3,4,5)$ of the normalized basis
in Table \ref{tab:ESnn} which do not depend on $g$.
This can be explained as follows. 
Each component $ES_6^1, ES_6^2$ of the Enomoto-Satoh trace map
$ES_6$ is a linear combination of various multi-contractions and hence, for each $g$,
the value of each element $v_i\ (i=1,2,3,4,5)$ is a constant multiple of 
the corresponding eigenvalue. Since the normalized vector $\bar{v}_i$ is
defined to be $v_i$ divided by the corresponding eigenvalue, the values of
$ES_6^1, ES_6^2$ on $\bar{v}_i$ are independent of $g$.

It was proved in \cite{mss4} that the former two elements in the theorem form a basis
of $\Im \tau_{g,1}(6)^{\mathrm{Sp}}\cong\Q^2$ for $g\geq 3$.
The case $g=2$ is shown similarly. We recall that, in this case, 
$\mathfrak{h}_{2,1}(1)\cong [1]$ and there is no $[1^3]$-summand.
On the other hand, it can be seen, from the values in Table \ref{tab:ESnn}, 
that the Enomoto-Satoh trace map vanishes on the last element. In fact, we have
\begin{align*}
ES_6^1(2\bar{v}_1+8\bar{v}_3+3\bar{v}_4)&=2 \cdot  6+8 \cdot 3+3\cdot (-12) =0\\
ES_6^2(2\bar{v}_1+8\bar{v}_3+3\bar{v}_4)&=2\cdot 6+8\cdot (-3)+3\cdot 4 =0.
\end{align*}
Also it is easy to see that this element is not a linear combination of the preceding 
elements. Hence these three elements, or these two elements when $g=2$, 
form a basis of $\Ker  ES_6$ as required.
\end{proof}

Henceforth, we call the last element 
$\varepsilon:=2\bar{v}_1+8\bar{v}_3+3\bar{v}_4$ the {\it extra} element
(of $\Ker  ES_6$).

\begin{table}[h]
\caption{$\text{Values of the Enomoto-Satoh trace maps $ES_6^i$ on the basis}$}
\begin{center}
\begin{tabular}{|c|r|r|r|r|r|}
\noalign{\hrule height0.8pt}
\hfil $g$ & $v_1$ 
& $v_2$ & $v_3$  & $v_4$
& $v_5$  \\
\hline
$1$ & $720$ & $0$ & $0$ & $0$ & $0$ \\  
{} & $720$ & $0$ & $0$ & $0$ & $0$ \\     
\hline
$2$ & $5040$ & $ 6480$ & $720$ & $-1440$ & $0$ \\  
{} & $5040$ & $-240$ & $-720$ & $480$ & $0$ \\     
\hline
$3$ & $18144$ & $36288$ & $4032$ & $-10080$ & $-1008$ \\  
{} & $18144$ & $-1344$ & $-4032$ & $3360$ & $336$ \\     
\hline
$4$ & $47520$ & $116640$ & $12960$ & $-36288$ & $-5184$ \\  
{} & $47520$ & $-4320$ & $-12960$ & $12096$ & $1728$ \\  
\hline
$5$ & $102960$ & $285120$ & $31680$ & $-95040$ & $ -15840$ \\  
{} & $102960$ & $-10560$ & $-31680$ & $31680$ & $5280$ \\    
\hline
$6$ & $196560$ & $589680$ & $65520$ & $ -205920$ & $ -37440$ \\  
{} & $196560$ & $-21840$ & $-65520$ & $68640$ & $12480$ \\     
\hline
$7$ & $342720$ & $ 1088640$ & $120960$ & $-393120$ & $-75600$ \\  
{} & $342720$ & $-40320$ & $-120960$ & $131040$ & $25200$ \\                                      
\noalign{\hrule height0.8pt}
\end{tabular}
\end{center}
\label{tab:ES}
\end{table}

\begin{table}[h]
\caption{$\text{Orthogonal decomposition of $\mathfrak{h}_{g,1}(6)^{\mathrm{Sp}}$\  }$}
\begin{center}
\begin{tabular}{|c|c|c|c|}
\noalign{\hrule height0.8pt}
\hfil $\lambda$ & $\mu_{\lambda'}\ \text{(eigen value of}\ H_{\lambda})$ 
& \text{$\dim\, H_{\lambda}$}  & \text{generators for $H_\lambda$} \\
\hline
$[4]$ & $2g(2g+1)(2g+2)(2g+3)$ & $1$ & $v_1$ \\
\hline
$[31]$ & $(2g-2)2g(2g+1)(2g+2)$ & $2$ & $v_2, v_3$ \\
\hline
$[2^2]$ & $(2g-2)(2g-1)2g(2g+1)$ & $1$ & $v_4$ \\
\hline
$[21^2]$ & $(2g-4)(2g-2)2g(2g+1)$ & $1$ & $v_5$ \\                                                          
\hline 
$\text{total}$ & {}   & $5$ & {} \\
\noalign{\hrule height0.8pt}
\end{tabular}
\end{center}
\label{tab:ev}
\end{table}

\begin{table}[h]
\caption{$\text{Values of the Enomoto-Satoh trace maps $ES_6^i$ on the normalized basis}$}
\begin{center}
\begin{tabular}{|r|r|r|r|r|r|}
\noalign{\hrule height0.8pt}
\hfil ${}$ & $\bar{v}_1$ 
& $\bar{v}_2$ & $\bar{v}_3$  & $\bar{v}_4$
& $\bar{v}_5;\ g\geq 3, g=2$  \\
\hline
$ES_6^1$ & $6$ & $27$ & $3$ & $-12$ & $-3,\hspace{10mm} 0$ \\
\hline
$ES_6^2$ & $6$ & $-1$ & $-3$ & $4$ &  $1,\hspace{10mm} 0$ \\                                                       
\noalign{\hrule height0.8pt}
\end{tabular}
\end{center}
\label{tab:ESnn}
\end{table}

\section{Structure of the $[1^3]$- and $[1]$-isotypical components $\mathfrak{h}_{g,1}(7)_{[1^3]}$}\label{sec:1_13}

Recall from Johnson's work \cite{j} that 
\[\Im\tau_{g,1}(1)=\mathfrak{h}_{g,1}(1)=\wedge^3 H\]
and the $\mathrm{Sp}$-irreducible decomposition of this module is given by 
$[1^3]\oplus [1]$. Fix a symplectic basis $\{x_1, y_1, \ldots, x_g, y_g\}$. Then 
\begin{align*}
x_1\wedge x_2\wedge x_3&=
x_1\otimes x_2\otimes x_3+x_2\otimes x_3\otimes x_1+x_3\otimes x_1\otimes x_2\\
&-x_1\otimes x_3\otimes x_2-x_2\otimes x_1\otimes x_3-x_3\otimes x_2\otimes x_1
\in \wedge^3 H \subset H^{\otimes 3}
\end{align*}
is a highest weight vector of $[1^3]\subset \mathfrak{h}_{g,1}(1)\subset H^{\otimes 3}$ and 
\[x_1\wedge \omega_0
=x_1\wedge (x_2\wedge y_2+\cdots + x_g\wedge y_g)
\in \wedge^3 H\]
is a highest weight vector of $[1]\subset \mathfrak{h}_{g,1}(1)$.

Now 
we analyze the $[1^3]$- and $[1]$- isotypical components of 
$\mathfrak{h}_{g,1}(7)$
in order to analyze the failure of the extra element $\varepsilon \in \Ker  ES_6$ constructed in the previous section 
to normalize the Johnson image: 
\[[\Im \tau_{g,1}(1), \varepsilon]\not\subset \Im\tau_{g,1}(7).\] 

We showed in \cite{mss4} that
\begin{align*}
\mathfrak{h}_{g,1}(7)&= 10 [1]\oplus 15 [1^3]\oplus \text{other terms},\\
\mathfrak{j}_{g,1}(7)&=3 [1]\oplus 3 [1^3]\oplus \text{other terms}
\end{align*}
and constructed various elements in $\Im \tau_{g,1}(7)$ 
as well as detectors for the $[1^3]$- and $[1]$-isotypical components of $\mathfrak{h}_{g,1}(7)$. 
Then by computing the values of the latter on the former elements
explicitly, we found that
$$
6 [1]\oplus 12 [1^3]\subset \Im \tau_{g,1}(7).
$$ 
Since we know that $\Im \tau_{g,1}(7)$ embeds in the 
quotient $\mathfrak{h}_{g,1}(7)/\mathfrak{j}_{g,1}(7)$, we can conclude that
\begin{align*}
&\text{the $[1^3]$-isotypical
component of $\Im \tau_{g,1}(7)$ is $12 [1^3]$}\\
&\text{the $[1]$-isotypical
component of $\Im \tau_{g,1}(7)$ is $\text{$6 [1]$ or $7 [1]$}$}
\end{align*}
Next we consider the Enomoto-Satoh trace map
\[ES_7 \colon \mathfrak{h}_{g,1}(7)\rightarrow \mathfrak{a}_g(5)=15 [1]\oplus 15 [1^3]
\oplus \text{other terms}.\]
We find that 
$$
\Im ES_7\cong3 [1]\oplus 2 [1^3]
$$
and the following five homomorphisms
\begin{align*}
W_1&=1\wedge 2\wedge 4\ (3 6)(5 7),\ W_2=1\wedge 2\wedge 4\ (3 7)(5 6),\\
F_1&=(13)(24)(56)\ 7,\ F_2=(13)(25)(46)\ 7,\ F_3=(13)(26)(45)\ 7
\end{align*}
detect them. Here $W_1$ and $F_1$, for example, are defined by
\begin{align*}
W_1(u_1\otimes\cdots\otimes u_7)=& (u_3\cdot u_6)(u_5\cdot u_7) u_1\wedge u_2\wedge u_4,\\
F_1(u_1\otimes\cdots\otimes u_7)=& (u_1\cdot u_3)(u_2\cdot u_5)(u_4\cdot u_6) u_7
\end{align*}
for $u_1\otimes\cdots\otimes u_7\in \mathfrak{a}_g(5)=(H^{\otimes 7})_{\Z/7\Z} \cong 
(H^{\otimes 7})^{\Z/7\Z} \subset H^{\otimes 7}$.
As mentioned above, $\mathfrak{j}_{g,1}(7)$ contains three copies of each of $[1]$ and $[1^3]$.
As for $[1]$, it turns out that $ES_7$ detects only two of the three copies of 
$[1]$ in $\mathfrak{j}_{g,1}(7)$. 
As for $[1^3]$,
it detects only two of the three copies of $[1^3]$ in $\mathfrak{j}_{g,1}(7)$. 
We now conclude that 
\[\text{the $[1]$-isotypical component of $\Im \tau_{g,1}(7)$ is $6 [1]$}\]
and we obtain the exact sequence
\[0\longrightarrow(\Im \tau_{g,1}(7))_1\cong 6 [1]\oplus 12 [1^3]
\longrightarrow (\Ker ES_7)_1 \longrightarrow [1]\oplus  [1^3] \longrightarrow 0,\]
where $(\cdot)_1$ denotes the $([1]\oplus [1^3])$-isotypical components.
Namely, the kernel $\Ker ES_7$ of the Enomoto-Satoh trace map in degree
$7$ contains contains one additional copy of $[1]$ and one additional copy of $[1^3]$ 
beyond those contained in the Johnson image.
Furthermore, these additional components are contained in $\mathfrak{j}_{g,1}(7)$.

It follows that there exist unique (up to nonzero scalar multiples) 
$\mathrm{Sp}$-homomorphisms 
\[R \colon \Ker ES_7\longrightarrow [1],\quad S \colon \Ker ES_7\longrightarrow [1^3]\]
which detect the above additional components, namely they vanish on the Johnson image and restrict to 
isomorphisms on the additional components.

\begin{prop}
The $\mathrm{Sp}$-homomorphisms $R$ and $S$ above extended to the whole of 
$H^{\otimes 9} \supset \mathfrak{h}_{g,1}(7)\supset \Ker ES_7$ as follows.
\begin{align*}
R(u_1\otimes\cdots\otimes u_{9})=&-2 (u_1\cdot u_2)(u_3\cdot u_5)(u_4\cdot u_7)(u_6\cdot u_8)\ u_9\\
&+2 (u_1\cdot u_2)(u_3\cdot u_6)(u_4\cdot u_7)(u_5\cdot u_8)\ u_9\\
&+2 (u_1\cdot u_2)(u_3\cdot u_6)(u_4\cdot u_8)(u_5\cdot u_7)\ u_9\\
&+2 (u_1\cdot u_2)(u_3\cdot u_7)(u_4\cdot u_6)(u_5\cdot u_8)\ u_9\\
&+ (u_1\cdot u_3)(u_2\cdot u_6)(u_4\cdot u_7)(u_5\cdot u_8)\ u_9\\
S(u_1\otimes\cdots\otimes u_{9})=&(u_2\cdot u_3)(u_6\cdot u_8)(u_7\cdot u_9)\ u_1\wedge u_4\wedge u_5\\
&-(u_2\cdot u_3)(u_4\cdot u_5)(u_6\cdot u_8)\ u_1\wedge u_7\wedge u_9\\
&+(u_4\cdot u_5)(u_6\cdot u_9)(u_7\cdot u_8)\ u_1\wedge u_2\wedge u_3\\
&-(u_2\cdot u_3)(u_4\cdot u_5)(u_6\cdot u_9)\ u_1\wedge u_7\wedge u_8\\
&+3(u_2\cdot u_3)(u_5\cdot u_7)(u_8\cdot u_9)\ u_1\wedge u_4\wedge u_6\\
&-4(u_2\cdot u_3)(u_4\cdot u_6)(u_5\cdot u_8)\ u_1\wedge u_7\wedge u_9\\
&-(u_2\cdot u_3)(u_4\cdot u_7)(u_5\cdot u_8)\ u_1\wedge u_6\wedge u_9\\
&+(u_2\cdot u_4)(u_3\cdot u_7)(u_5\cdot u_8)\ u_1\wedge u_6\wedge u_9
\end{align*}
\label{prop:rs}
\end{prop}

\begin{proof}[Sketch of Proof]
By looking at the elements of the $([1]\oplus [1^3])$-isotypical components
of $\Im \tau_{g,1}(7)$ and computing as well the explicit values of the map $ES_7$ on $\mathfrak{h}_{g,1}(7)$,
we obtain the formulae for the homomorphisms $R, S$ with the required properties.
\end{proof}

We summarize the above discussion of the structure of the $([1]\oplus [1^3])$-isotypical 
components of $\mathfrak{h}_{g,1}(7)$ as follows.

\begin{prop}[Structure of $\mathfrak h_{g,1}(7)_1$]
We have the following exact sequences.
$$
0\longrightarrow(\Ker ES_7)_1\cong 7 [1]\oplus 13 [1^3] \longrightarrow 
\mathfrak{h}_{g,1}(7)_1 \xrightarrow{\mathrm{ES}_7} 3 [1]\oplus  2[1^3] \longrightarrow 0,
$$
$$
0\longrightarrow(\mathrm{Im}\,\tau_{g,1}(7))_1\cong 6 [1]\oplus 12 [1^3]
\longrightarrow
(\Ker ES_7)_1 \xrightarrow{(R,S)} [1]\oplus  [1^3]
\longrightarrow 0, 
$$
$$
0\longrightarrow \mathfrak{j}_{g,1}(7)_1\cap (\Ker ES_7)_1\cong  [1]\oplus  [1^3]
\longrightarrow
\mathfrak{j}_{g,1}(7)_1 \xrightarrow{\mathrm{ES}_7} 2 [1]\oplus  2[1^3]
\longrightarrow 0.
$$
\label{prop:str1}
\end{prop}

Now the Galois obstruction $\sigma_3(g)\in \mathfrak{h}_{g,1}(6)^{\mathrm{Sp}}$ 
normalizes the Johnson
image:
$$
[\mathrm{Im} \tau_{g,1}(1), \sigma_3(g)]\subset \Im \tau_{g,1}(7).
$$ 
To find such an element, we use the description of $\mathfrak{h}_{g,1}(7)_1$ 
given in Proposition \ref{prop:str1} and make the following definition.

\begin{definition}\label{def:homo}
We define homomorphisms 
\[w_1,w_2, f_1, f_2, f_3, r, s \colon \mathfrak{h}_{g,1}(6)^{\mathrm{Sp}}\longrightarrow \Q\]
as follows.
\begin{align*}
w_i(v)&=\text{coefficient of $x_1\wedge x_2\wedge x_3$ of}\ W_i([x_1
\wedge x_2\wedge x_3, v]) \quad (i=1,2), \\
f_i(v)&=\text{coefficient of $x_1$ of}\ F_i([x_1
\wedge x_2\wedge y_2, v]) \quad (i=1,2,3),\\
r(v)&=\text{coefficient of $x_1$ of}\ R([x_1
\wedge x_2\wedge y_2, v]), \\
s(v)&=\text{coefficient of $x_1\wedge x_2\wedge x_3$ of}\ S([x_1
\wedge x_2\wedge x_3, v]). 
\end{align*}
\end{definition}

\begin{remark}\label{remark:1}
It would be more natural to replace $r$ by
\[\tilde{r}(v)=\text{coefficient of $x_1$ of}\ R([x_1 \wedge \omega_0, v])\]
because $x_1\wedge \omega_0$ is the highest weight vector of 
the irreducible component $[1]\subset \mathfrak{h}_{g,1}(1)$.
However, it can be seen that
\[\tilde{r}(v)=(g-1) r(v)\]
for any $v$, so that $r$ has the same information as $\tilde{r}$.
This is because the symplectic transformation induced by interchanging $x_2, y_2$
and $x_k,y_k$ induces an isomorphism of the Lie algebra $\mathfrak{h}_{g,1}$
for any $k=3,4,\ldots,g$. Hence, in the above definition of $r$, even if we replace 
$x_1\wedge x_2\wedge y_2$ by $x_1\wedge x_k\wedge y_k$, we obtain the same function.
Since
$$
x_1\wedge \omega_0=\sum_{k=2}^g x_1\wedge x_k\wedge y_k,
$$
we have the required equality $\tilde{r}=(g-1) r$. We use $r$ rather than $\tilde{r}$ because
it is easier to compute the former than the latter.
\end{remark}

Using Mathematica, we computed the values of the above seven homomorphisms 
$w_i \ (i=1,2), f_i\ (i=1,2,3), r$ and $s$ 
for the cases of genus $g=3,4,5,6,7$. 
The results are described in Tables \ref{tab:g3}--\ref{tab:g7}.

\begin{table}[h]
\caption{$\text{Detecting various $[1^3]$-, $[1]$-summands in $[\mathfrak{h}_{3,1}(1),v_i]$}$}
\begin{center}
\begin{tabular}{|c|c|c|c|c|c|c|c|}
\noalign{\hrule height0.8pt}
\hfil $g=3$ & $w_1$ 
& $w_2$ & $f_1$  & $f_2$
& $f_3$ & $r$ & $s$ \\
\hline
$v_1$ & $0$ & $2160\cdot 6$ & $ -6048$ & $0$ & $0$ & $-6048$   & $0$\\
\hline
$v_2$ & $0$ & $-1280\cdot 6$ & $3584$ & $0$ &  $0$ & $9072$ & $3136\cdot 6$\\
\hline
$v_3$ & $0$ & $ -720\cdot 6$ & $2016$ & $0$ &  $0$ & $ -1008$ & $-1728\cdot 6$\\
\hline
$v_4$ & $0$ & $800\cdot 6$ & $-2240$ & $0$ &  $0$ & $3780$ & $3440\cdot 6$\\
\hline
$v_5$ & $0$ & $ 80\cdot 6$ & $-224$  & $0$ &  $0$ & $1260$ & $848\cdot 6$\\                                                               
\noalign{\hrule height0.8pt}
\end{tabular}
\end{center}
\label{tab:g3}
\end{table}

\begin{table}[h]
\caption{$\text{Detecting various $[1^3]$-, $[1]$-summands in $[\mathfrak{h}_{4,1}(1),v_i]$}$}
\begin{center}
\begin{tabular}{|c|c|c|c|c|c|c|c|}
\noalign{\hrule height0.8pt}
\hfil $g=4$ & $w_1$ 
& $w_2$ & $f_1$  & $f_2$
& $f_3$ & $r$ & $s$ \\
\hline
$v_1$ & $0$ & $4620\cdot 6$ & $ -11880$ & $0$ & $0$ & $-11880$   & $0$\\
\hline
$v_2$ & $0$ & $-2380\cdot 6$ & $6120$ & $0$ &  $0$ & $18900$ & $7100\cdot 6$\\
\hline
$v_3$ & $0$ & $-1680\cdot 6$ & $4320$ & $0$ &  $0$ & $ -1620$ & $-3300\cdot 6$\\
\hline
$v_4$ & $0$ & $1960\cdot 6$ & $-5040$ & $0$ &  $0$ & $8316$ & $7420\cdot 6$\\
\hline
$v_5$ & $0$ & $ 280\cdot 6$ & $-720$  & $0$ &  $0$ & $4104$ & $2680\cdot 6$\\                                                               
\noalign{\hrule height0.8pt}
\end{tabular}
\end{center}
\label{tab:g4}
\end{table}

\begin{table}[h]
\caption{$\text{Detecting various $[1^3]$-, $[1]$-summands in $[\mathfrak{h}_{5,1}(1),v_i]$}$}
\begin{center}
\begin{tabular}{|c|c|c|c|c|c|c|c|}
\noalign{\hrule height0.8pt}
\hfil $g=5$ & $w_1$ 
& $w_2$ & $f_1$  & $f_2$
& $f_3$ & $r$ & $s$ \\
\hline
$v_1$ & $0$ & $8424\cdot 6$ & $ -20592$ & $0$ & $0$ & $-20592$   & $0$\\
\hline
$v_2$ & $0$ & $-3888\cdot 6$ & $9504$ & $0$ &  $0$ & $34056$ & $13392\cdot 6$\\
\hline
$v_3$ & $0$ & $-3240\cdot 6$ & $7920$ & $0$ &  $0$ & $-2376$ & $-5616\cdot 6$\\
\hline
$v_4$ & $0$ & $3888\cdot 6$ & $-9504$ & $0$ &  $0$ & $15444$ & $13608\cdot 6$\\
\hline
$v_5$ & $0$ & $ 648\cdot 6$ & $ -1584$  & $0$ &  $0$ & $9108$ & $5832\cdot 6$\\                                                               
\noalign{\hrule height0.8pt}
\end{tabular}
\end{center}
\label{tab:g5}
\end{table}

\begin{table}[h]
\caption{$\text{Detecting various $[1^3]$-, $[1]$-summands in $[\mathfrak{h}_{6,1}(1),v_i]$}$}
\begin{center}
\begin{tabular}{|c|c|c|c|c|c|c|c|}
\noalign{\hrule height0.8pt}
\hfil $g=6$ & $w_1$ 
& $w_2$ & $f_1$  & $f_2$
& $f_3$ & $r$ & $s$ \\
\hline
$v_1$ & $0$ & $13860\cdot 6$ & $-32760$ & $0$ & $0$ & $ -32760$   & $0$\\
\hline
$v_2$ & $0$ & $-5852\cdot 6$ & $13832$ & $0$ &  $0$ & $55692$ & $22540\cdot 6$\\
\hline
$v_3$ & $0$ & $ -5544\cdot 6$ & $13104$ & $0$ &  $0$ & $ -3276$ & $ -8820\cdot 6$\\
\hline
$v_4$ & $0$ & $6776\cdot 6$ & $ -16016$ & $0$ &  $0$ & $25740$ & $22484\cdot 6$\\
\hline
$v_5$ & $0$ & $1232\cdot 6$ & $-2912$  & $0$ &  $0$ & $16848$ & $10640\cdot 6$\\                                                               
\noalign{\hrule height0.8pt}
\end{tabular}
\end{center}
\label{tab:g6}
\end{table}

\begin{table}[h]
\caption{$\text{Detecting various $[1^3]$-, $[1]$-summands in $[\mathfrak{h}_{7,1}(1),v_i]$}$}
\begin{center}
\begin{tabular}{|c|c|c|c|c|c|c|c|}
\noalign{\hrule height0.8pt}
\hfil $g=7$ & $w_1$ 
& $w_2$ & $f_1$  & $f_2$
& $f_3$ & $r$ & $s$ \\
\hline
$v_1$ & $0$ & $21216\cdot 6$ & $-48960$ & $0$ & $0$ & $  -48960$   & $0$\\
\hline
$v_2$ & $0$ & $ -8320\cdot 6$ & $19200$ & $0$ &  $0$ & $84960$ & $ 35072\cdot 6$\\
\hline
$v_3$ & $0$ & $ -8736\cdot 6$ & $20160$ & $0$ &  $0$ & $-4320$ & $ -13056\cdot 6$\\
\hline
$v_4$ & $0$ & $10816\cdot 6$ & $-24960$ & $0$ &  $0$ & $39780$ & $34528\cdot 6$\\
\hline
$v_5$ & $0$ & $2080\cdot 6$ & $ -4800$  & $0$ &  $0$ & $27900$ & $17440\cdot 6$\\                                                               
\noalign{\hrule height0.8pt}
\end{tabular}
\end{center}
\label{tab:g7}
\end{table}

\newpage
\section{Polynomiality of the homomorphisms $r,s,f_i$ and $w_i$}\label{sec:polyn}

\begin{table}[h]
\caption{$\text{Polynomials of the homomorphisms $r,w_2$ on the basis}$}
\begin{center}
\begin{tabular}{|c|c|c|}
\noalign{\hrule height0.8pt}
\hfil $\text{basis}$ & $r$ 
& $w_2$ \\
\hline
$v_1$ & $-24(g+1)(2g+1)(2g+3)$ & $12(g+1)(2g-1)(2g+3)$  \\      
\hline
$v_2$ & $12(g+1)(2g+1)(8g+3)$ & $-4(g+1)(g+13)(2g-1)$  \\  
\hline
$v_3$ & $-36(g+1)(2g+1)$ & $-12g(g+1)(2g-1)$  \\
\hline
$v_4$ & $12(2g-1)(2g+1)(2g+3)$ & $8(g+1)(2g-1)^2$  \\
\hline
$v_5$ & $12(g-2)(2g+1)(4g+3)$ & $4(g-2)(g+1)(2g-1)$  \\                                    
\noalign{\hrule height0.8pt}
\end{tabular}
\end{center}
\label{tab:rw}
\end{table}

\begin{table}[h]
\caption{$\text{Polynomials of the homomorphisms $s,f_1$ on the basis}$}
\begin{center}
\begin{tabular}{|c|c|c|}
\noalign{\hrule height0.8pt}
\hfil $\text{basis}$ & $s$ 
& $f_1$ \\
\hline
$v_1$ & $0$ & $-24(g+1)(2g+1)(2g+3)$  \\      
\hline
$v_2$ & $4(g+1)^2(22g-17)$ & $8(g+1)(g+13)(2g+1)$  \\  
\hline
$v_3$ & $-12(g+1)^2(2g+3)$ & $24g(g+1)(2g+1)$  \\
\hline
$v_4$ & $4(g+1)(2g-1)(10g+13)$ & $-16(g+1)(2g-1)(2g+1)$  \\
\hline
$v_5$ & $4(g-2)(g+1)(14g+11)$ & $-8(g-2)(g+1)(2g+1)$  \\                                    
\noalign{\hrule height0.8pt}
\end{tabular}
\end{center}
\label{tab:sf}
\end{table}

In what follows, we always identify the spaces $h_{g,1}(6)^{\mathrm{Sp}}$ 
by the identifications described in Remark \ref{rem:inv_stable}. 
Thus, an element $v\in h_{g,1}(6)^{\mathrm{Sp}}$ 
is regarded as the same element for all genera under these identifications.
The purpose of this section is to prove the following proposition, 
postponing its proof until the end. 

\begin{prop}
The values of the homomorphisms $r,s,f_i,w_i$,
at any $v\in \mathfrak{h}_{g,1}^{\mathrm{Sp}} (6)$ are all polynomials in $g$
of degree at most $3$.
\label{prop:polynomial}
\end{prop}

For each homomorphism $h \in\{r,s,f_i,w_i\}$ and each $v \in \mathfrak h_{g,1}(6)^{\mathrm{Sp}}$, 
we call the polynomial $h(v)$, regarded as a function of $g$, the {\it polynomial of $h$ at $v$}.

\begin{cor}\label{cor:4g}
Each of the polynomials of the homomorphisms $r,s,f_i,w_i$ at $v$ 
is uniquely determined by its values at any four distinct genera $g_i\ (i=1,2,3,4)$,
e.g. $g=3,4,5,6$ or $g=4,5,6,7$.
\end{cor}

\begin{proof}[Proof of Corollary $\ref{cor:4g}$]
A polynomial in $g$ of at most degree $d$ is uniquely determined by its
values at any distinct $d+1$ values of $g$. This follows, for example, from 
the invertibility of the corresponding Vandermonde matrix.
\end{proof}

\begin{prop}
The homomorphisms $f_2,f_3,w_1$ vanish identically, and
the polynomials of the homomorphisms $r,s,f_1,w_2$ at the elements $v_i$ are given by 
Tables $\ref{tab:rw}$ and $\ref{tab:sf}$.
\label{prop:poly}
\end{prop}

\begin{proof}
This follows by applying Corollary \ref{cor:4g} to 
the data for any four of the genera listed in Tables \ref{tab:g3}--\ref{tab:g7}, 
for example $g=3,4,5,6$ or $g=4,5,6,7$.
As an example, consider $r(v_1)$. 
By Proposition \ref{prop:polynomial}, we can write
$$
r(v_1)(g)=a g^3+b g^2+c g+d
$$
for some constants $a,b,c,d$. Then we have
$$
\begin{pmatrix}
g_1^3 & g_2^3 & g_3^3 & g_4^3\\
g_1^2 & g_2^2 & g_3^2 & g_4^2\\
g_1 & g_2 & g_3 & g_4\\
1 & 1 & 1 & 1\\
\end{pmatrix}
\begin{pmatrix}
a\\
b\\
c\\
d\\
\end{pmatrix}
=
\begin{pmatrix}
r(v_1)(g_1)\\
r(v_1)(g_2)\\
r(v_1)(g_3)\\
r(v_1)(g_4)\\
\end{pmatrix}
$$
where $g_1,g_2,g_3,g_4$ are four values of $g$. If they are mutually distinct, then we have
$$
\begin{pmatrix}
a\\
b\\
c\\
d\\
\end{pmatrix}
=
\begin{pmatrix}
g_1^3 & g_2^3 & g_3^3 & g_4^3\\
g_1^2 & g_2^2 & g_3^2 & g_4^2\\
g_1 & g_2 & g_3 & g_4\\
1 & 1 & 1 & 1\\
\end{pmatrix}
^{-1}
\begin{pmatrix}
r(v_1)(g_1)\\
r(v_1)(g_2)\\
r(v_1)(g_3)\\
r(v_1)(g_4)\\
\end{pmatrix}
$$
Taking $g_1,g_2,g_3,g_4$ to be $3,4,5,6$, we obtain from Tables \ref{tab:g3}--\ref{tab:g6} that 
$$
\begin{pmatrix}
r(v_1)(3)\\
r(v_1)(4)\\
r(v_1)(5)\\
r(v_1)(6)\\
\end{pmatrix}
=
\begin{pmatrix}
-6048\\
-11880\\
-20592\\
-32760\\
\end{pmatrix}.
$$
Hence
$$
\begin{pmatrix}
a\\
b\\
c\\
d\\
\end{pmatrix}
=
\begin{pmatrix}
3^3 & 4^3 & 5^3 & 6^3\\
3^2 & 4^2 & 5^2 & 6^2\\
3 & 4 & 5 & 6\\
1 & 1 & 1 & 1\\
\end{pmatrix}
^{-1}
\begin{pmatrix}
-6048\\
-11880\\
-20592\\
-32760\\
\end{pmatrix}
=
\begin{pmatrix}
-96\\
-288\\
-264\\
-72\\
\end{pmatrix}.
$$
From this, we conclude that 
$$
r(v_1)(g)=-96g^3-288g^2-264g-72=-24(g+1)(2g+1)(2g+3).
$$
Alternatively, factoring the values of $r(v_1)$ as follows,
$$
\begin{pmatrix}
r(v_1)(3)\\
r(v_1)(4)\\
r(v_1)(5)\\
r(v_1)(6)\\
\end{pmatrix}
=
\begin{pmatrix}
-6048\\
-11880\\
-20592\\
-32760\\
\end{pmatrix}
=
\begin{pmatrix}
-24\cdot 4\cdot 7\cdot 9\\
-24\cdot 5\cdot 9\cdot 11\\
-24\cdot 6\cdot 11\cdot 13\\
-24\cdot 7\cdot 13\cdot 15\\
\end{pmatrix},
$$
we can read off the required polynomial without solving the above linear equation.
\end{proof}


We now establish several preliminary facts for the proof of Proposition \ref{prop:polynomial}.
First, we recall that the Lie bracket 
\[[\cdot ,\cdot ]: \mathfrak{h}_{g,1}(s-1)\otimes \mathfrak{h}_{g,1} (t-1)\longrightarrow 
\mathfrak{h}_{g,1} (s+t-2)\]
on $\mathfrak{h}_{g,1}$ is extended to an operation
\[[\cdot ,\cdot ]: H^{\otimes (s+1)}\otimes H^{\otimes (t+1)}\longrightarrow H^{\otimes (s+t)}\]
on the tensor algebra on $H$ defined as follows.
\begin{align*}
&[u_0\otimes u_1\otimes\cdots\otimes u_s, v_0\otimes v_1\otimes\cdots\otimes v_{t}]\\
=&-u_0\otimes \sum_{i=1}^s (v_0\cdot u_i)\, u_1\otimes \cdots \otimes u_{i-1}\otimes 
(v_1\otimes\cdots \otimes v_t)\otimes u_{i+1}\otimes\cdots\otimes u_s\\
&\hspace{3mm} +v_0\otimes \sum_{j=1}^t (u_0\cdot v_j)\, v_1\otimes \cdots \otimes v_{j-1}\otimes 
(u_1\otimes\cdots \otimes u_s)\otimes v_{j+1}\otimes\cdots\otimes v_t.
\end{align*}
Here $u_i, v_j\in H$ and $u\cdot v$ denotes the intersection number of $u$ and $v$.

\begin{lem}
We have the equality
\begin{align*}
[x_1\otimes x_2\otimes x_3, \omega_0&\otimes\omega_0\otimes\omega_0\otimes\omega_0]=
-x_1\otimes x_2\otimes \omega_0\otimes\omega_0\otimes\omega_0\otimes x_3\\
&+\omega_0\otimes (x_1\otimes x_2\otimes x_3-x_2\otimes x_3\otimes x_1)
\otimes \omega_0\otimes\omega_0\\
&+\omega_0\otimes \omega_0\otimes 
(x_1\otimes x_2\otimes x_3-x_2\otimes x_3\otimes x_1)\otimes \omega_0\\
&+\omega_0\otimes \omega_0\otimes \omega_0\otimes 
(x_1\otimes x_2\otimes x_3-x_2\otimes x_3\otimes x_1).
\end{align*}
\label{lem:123}
\end{lem}

\begin{proof}
By the definition of the bracket, we have
\begin{align*}
[x_1\otimes x_2\otimes x_3,&\, \omega_0\otimes\omega_0\otimes\omega_0\otimes\omega_0]=\\
&-x_1\otimes (x_2\otimes \omega_0\otimes\omega_0\otimes\omega_0\otimes x_3
+x_2\otimes x_3\otimes \omega_0\otimes\omega_0\otimes\omega_0)\\
&+x_1\otimes\left( x_2\otimes x_3\otimes \omega_0\otimes\omega_0\otimes\omega_0
+y_1\otimes Q  \right)\\
&(-y_1\otimes x_1+x_2\otimes y_2-y_2\otimes x_2+\cdots+x_g\otimes y_g-y_g\otimes x_g)\otimes Q,
\end{align*}
where
\begin{align*}
Q&= x_1\otimes x_2\otimes x_3\otimes \omega_0\otimes\omega_0
- x_2\otimes x_3\otimes x_1\otimes \omega_0\otimes\omega_0\\
&\quad + \omega_0\otimes x_1\otimes x_2\otimes x_3\otimes \omega_0
- \omega_0\otimes x_2\otimes x_3\otimes x_1\otimes\omega_0\\
&\quad + \omega_0\otimes \omega_0\otimes x_1\otimes x_2\otimes x_3
- \omega_0\otimes \omega_0\otimes x_2\otimes x_3\otimes x_1.\\
\end{align*}
Hence we have
\[[x_1\otimes x_2\otimes x_3, \ \omega_0\otimes\omega_0\otimes\omega_0\otimes\omega_0]=\\
-x_1\otimes x_2\otimes \omega_0\otimes\omega_0\otimes\omega_0\otimes x_3
+\omega_0\otimes Q.\]
Substituting the expression for $Q$ into the preceding formula 
gives the right-hand side of the claimed identity.
\end{proof}

\begin{lem}
For any permutation $\sigma=\begin{pmatrix}
1 & 2 & 3\\
\sigma_1&\sigma_2&\sigma_3
\end{pmatrix}$ of the three letters $1,2,3$, 
we have the equality
\begin{align*}
[x_{\sigma_1}\otimes x_{\sigma_2}\otimes x_{\sigma_3}, \omega_0&\otimes\omega_0\otimes\omega_0\otimes\omega_0]=
-x_{\sigma_1}\otimes x_{\sigma_2}\otimes \omega_0\otimes\omega_0\otimes\omega_0\otimes x_{\sigma_3}\\
&+\omega_0\otimes 
(x_{\sigma_1}\otimes x_{\sigma_2}\otimes x_{\sigma_3}
-x_{\sigma_2}\otimes x_{\sigma_3}\otimes x_{\sigma_1})
\otimes \omega_0\otimes\omega_0\\
&+\omega_0\otimes \omega_0\otimes 
(x_{\sigma_1}\otimes x_{\sigma_2}\otimes x_{\sigma_3}
-x_{\sigma_2}\otimes x_{\sigma_3}\otimes x_{\sigma_1})
\otimes \omega_0\\
&+\omega_0\otimes \omega_0\otimes \omega_0\otimes 
(x_{\sigma_1}\otimes x_{\sigma_2}\otimes x_{\sigma_3}
-x_{\sigma_2}\otimes x_{\sigma_3}\otimes x_{\sigma_1})
\end{align*}
\label{lem:s123}
\end{lem}

\begin{proof}
Computations similar to those in the proof of Lemma \ref{lem:123} 
yield the required identity.
Alternatively, apply the index permutation $\sigma$ to the equality in Lemma \ref{lem:123}, 
noting that $\omega_0$ is invariant under permutations of the basis indices. 
\end{proof}

\begin{prop}
We have the equality
\[[x_1\wedge x_2\wedge x_3, \omega_0\otimes\omega_0\otimes\omega_0\otimes\omega_0]=- 
\sigma (x_1\wedge x_2\wedge x_3\otimes\omega_0\otimes\omega_0\otimes\omega_0),\]
where $\sigma$ is the permutation operator on $H^{\otimes9}$ defined by
\[\sigma(u_1\otimes u_2\otimes u_3\otimes\omega_0\otimes\omega_0\otimes\omega_0)
=u_1\otimes u_2\otimes\omega_0\otimes\omega_0\otimes\omega_0\otimes u_3.\]
\label{prop:123om}
\end{prop}

\begin{proof}
Recall that the element $x_1\wedge x_2\wedge x_3$ is expressed as 
\begin{align*}
x_1\wedge x_2\wedge x_3=&x_1\otimes x_2\otimes x_3+x_2\otimes x_3\otimes x_1+x_3\otimes x_1\otimes x_2\\
-&x_2\otimes x_1\otimes x_3-x_3\otimes x_2\otimes x_1-x_1\otimes x_3\otimes x_2.
\end{align*}
\noindent
Summing the first terms on the right-hand sides of Lemmas \ref{lem:123} and \ref{lem:s123} 
corresponding to the six terms above, we obtain exactly the expression in the statement.
Therefore it remains to show that the remaining terms sum up to zero.
The second term of the equality in Lemma \ref{lem:123} for the first component 
$x_1\otimes x_2\otimes x_3$ of $x_1\wedge x_2\wedge x_3$ is
$$
\omega_0\otimes (x_1\otimes x_2\otimes x_3-x_2\otimes x_3\otimes x_1)
\otimes \omega_0\otimes\omega_0.
$$
Adding the corresponding second terms from Lemma \ref{lem:s123} for 
the remaining five terms of $x_1\wedge x_2\wedge x_3$, we obtain
\begin{align*}
&\omega_0\otimes (x_1\otimes x_2\otimes x_3-x_2\otimes x_3\otimes x_1)
\otimes \omega_0\otimes\omega_0\\
+&\omega_0\otimes (x_2\otimes x_3\otimes x_1-x_3\otimes x_1\otimes x_2)
\otimes \omega_0\otimes\omega_0\\
+&\omega_0\otimes (x_3\otimes x_1\otimes x_2-x_1\otimes x_2\otimes x_3)
\otimes \omega_0\otimes\omega_0\\
+&\omega_0\otimes (-x_2\otimes x_1\otimes x_3+x_1\otimes x_3\otimes x_2)
\otimes \omega_0\otimes\omega_0\\
+&\omega_0\otimes (-x_3\otimes x_2\otimes x_1+x_2\otimes x_1\otimes x_3)
\otimes \omega_0\otimes\omega_0\\
+&\omega_0\otimes (-x_1\otimes x_3\otimes x_2+x_3\otimes x_2\otimes x_1)
\otimes \omega_0\otimes\omega_0=0.\\
\end{align*}
Similarly, the corresponding third and fourth terms in 
Lemma \ref{lem:123} and Lemma \ref{lem:s123} sum up to zero. 
This completes the proof.
\end{proof}

Now we generalize Lemma \ref{lem:123}, Lemma \ref{lem:s123} 
and Proposition \ref{prop:123om} by replacing 
$\omega_0\otimes \omega_0\otimes\omega_0\otimes\omega_0$
with the tensor associated with 
an arbitrary linear chord diagram 
\begin{align*}
C=&\{\{i_1,j_1\},\{i_2,j_2\},\{i_3,j_3\},\{i_4,j_4\}\}\\
1=i_1&<i_2<i_3<i_4\leq 7,\quad 
 i_k <j_k\ (k=1,2,3,4)
\end{align*}
consisting of four chords. We define
$$
C_1\tilde{\otimes}C_2\tilde{\otimes}C_3\tilde{\otimes}C_4\in (H^{\otimes 8})^{\mathrm{Sp}}
$$
to be the tensor product of four symplectic elements
$$
C_k=\omega_0(k)\in H^{(i_k)}\otimes H^{(j_k)}\quad (k=1,2,3,4)
$$
where $H^{(r)}$ denotes the $r$-th component of $H^{\otimes 8}$
and the symbol $\tilde{\otimes}$ denotes the tensor product respecting 
the order of components.

\begin{lem}
We have the equality
\begin{align*}
[x_1\otimes x_2\otimes x_3, C_1&\tilde{\otimes}C_2\tilde{\otimes}C_3\tilde{\otimes}C_4]=
-x_1\otimes \{\check{i}_1\tilde{\otimes}j_1(x_2) 
\tilde{\otimes}C_2\tilde{\otimes}C_3\tilde{\otimes}C_4\}\otimes x_3\\
&-x_1\otimes x_2\otimes \{\check{i}_1\tilde{\otimes}j_1(x_3) 
\tilde{\otimes}C_2\tilde{\otimes}C_3\tilde{\otimes}C_4\}\\
&+\{i_1(x_1)\tilde{\otimes}j_1(x_2\otimes x_3) 
\tilde{\otimes}C_2\tilde{\otimes}C_3\tilde{\otimes}C_4\}\\
&+\{C_1\tilde{\otimes} \left(i_2(x_1)\tilde{\otimes}j_2(x_2\otimes x_3)
-i_2(x_2\otimes x_3)\tilde{\otimes}j_2(x_1)\right)
\tilde{\otimes}C_3 \tilde{\otimes}C_4\}\\
&+\{C_1\tilde{\otimes}C_2\tilde{\otimes} \left(i_3(x_1)\tilde{\otimes}j_3(x_2\otimes x_3)
-i_3(x_2\otimes x_3)\tilde{\otimes}j_3(x_1)\right)
 \tilde{\otimes}C_4\}\\
&+\{C_1\tilde{\otimes}C_2 \tilde{\otimes}C_3\tilde{\otimes} 
\left(i_4(x_1)\tilde{\otimes}j_4(x_2\otimes x_3)
-i_4(x_2\otimes x_3)\tilde{\otimes}j_4(x_1)\right)
\}.
\end{align*}
Here in the notation 
$\{\check{i}_1\tilde{\otimes}j_1(x_2) 
\tilde{\otimes}C_2\tilde{\otimes}C_3\tilde{\otimes}C_4\}$, 
for example, the symbol $\check{i}_1$ indicates that the $i_1$-st component 
is deleted, while $j_1(x_2)$ indicates that $x_2$ is inserted into the original $j_1$-st component
of $H^{\otimes 8}$.
Also the symbol $\{\ \ \}$ indicates that all positions are counted in the original tensor product $H^{\otimes8}$. 
After carrying out the indicated deletions and insertions, the factors written before and after $\{\ \ \}$ 
are then appended to form the resulting tensor.
Thus
$$
x_1\otimes \{\check{i}_1\tilde{\otimes}j_1(x_2) 
\tilde{\otimes}C_2\tilde{\otimes}C_3\tilde{\otimes}C_4\}\otimes x_3
\in H\otimes H^{\otimes 7}\otimes H=H^{\otimes 9}.
$$
\label{lem:g123}
\end{lem}

\begin{proof}
By the definition of the bracket, we have
\begin{align*}
[x_1\otimes x_2\otimes x_3, \{C_1\tilde{\otimes}C_2\tilde{\otimes}C_3\tilde{\otimes}C_4\}]=&\,
-x_1\otimes \{\check{i}_1\tilde{\otimes}j_1(x_2) 
\tilde{\otimes}C_2\tilde{\otimes}C_3\tilde{\otimes}C_4\}\otimes x_3\\
&-x_1\otimes x_2\otimes \{\check{i}_1\tilde{\otimes}j_1(x_3) 
\tilde{\otimes}C_2\tilde{\otimes}C_3\tilde{\otimes}C_4\}\\
&+\{i_1(x_1)\tilde{\otimes}j_1(x_2\otimes x_3) 
\tilde{\otimes}C_2\tilde{\otimes}C_3\tilde{\otimes}C_4\}\\
&+\{C_1\tilde{\otimes} (D(x_1)\otimes x_2\otimes x_3)(C_2\tilde{\otimes}C_3\tilde{\otimes}C_4)\},
\end{align*}
where the operator $(D(x_1)\otimes x_2\otimes x_3)$ is defined by
\[
(D(x_1)\otimes x_2\otimes x_3)(u_1\otimes\cdots\otimes u_r)\\
=\sum_{k=1}^r u_1\otimes\cdots\otimes u_{k-1}\otimes (x_1\cdot u_k (x_2\otimes x_3))\otimes u_{k+1}
\otimes\cdots\otimes u_r.
\]
On the other hand, within $H^{\otimes 8}$, we have
\[(D(x_1)\otimes x_2\otimes x_3)\ C_k=
i_k(x_1)\tilde{\otimes}j_k(x_2\otimes x_3)
-i_k(x_2\otimes x_3)\tilde{\otimes}j_k(x_1).\]
Substituting this formula into the preceding expression gives the required result.
\end{proof}

\begin{lem}\label{lem:gs123}
For any permutation $\sigma=\begin{pmatrix}
1 & 2 & 3\\
\sigma_1&\sigma_2&\sigma_3
\end{pmatrix}$, we have the equality
\begin{align*}
[x_{\sigma_1}\otimes x_{\sigma_2}\otimes x_{\sigma_3}, & \,\{C_1
\tilde{\otimes}C_2\tilde{\otimes}C_3\tilde{\otimes}C_4\}]=
-x_{\sigma_1}\otimes \{\check{i}_1\tilde{\otimes}j_1(x_{\sigma_2}) 
\tilde{\otimes}C_2\tilde{\otimes}C_3\tilde{\otimes}C_4\}\otimes x_{\sigma_3}\\
&-x_{\sigma_1}\otimes x_{\sigma_2}\otimes \{\check{i}_1\tilde{\otimes}j_1(x_{\sigma_3}) 
\tilde{\otimes}C_2\tilde{\otimes}C_3\tilde{\otimes}C_4\}\\
&+\{i_1(x_{\sigma_1})\tilde{\otimes}j_1(x_{\sigma_2}\otimes x_{\sigma_3}) 
\tilde{\otimes}C_2\tilde{\otimes}C_3\tilde{\otimes}C_4\}\\
&+\{C_1\tilde{\otimes} \left(i_2(x_{\sigma_1})\tilde{\otimes}j_2(x_{\sigma_2}\otimes x_{\sigma_3})
-i_2(x_{\sigma_2}\otimes x_{\sigma_3})\tilde{\otimes}j_2(x_{\sigma_1})\right)
\tilde{\otimes}C_3 \tilde{\otimes}C_4\}\\
&+\{C_1\tilde{\otimes}C_2\tilde{\otimes} \left(i_3(x_{\sigma_1})\tilde{\otimes}j_3(x_{\sigma_2}\otimes x_{\sigma_3})
-i_3(x_{\sigma_2}\otimes x_{\sigma_3})\tilde{\otimes}j_3(x_{\sigma_1})\right)
 \tilde{\otimes}C_4\}\\
&+\{C_1\tilde{\otimes}C_2 \tilde{\otimes}C_3\tilde{\otimes} 
\left(i_4(x_{\sigma_1})\tilde{\otimes}j_4(x_{\sigma_2}\otimes x_{\sigma_3})
-i_4(x_{\sigma_2}\otimes x_{\sigma_3})\tilde{\otimes}j_4(x_{\sigma_1})\right)
\}.
\end{align*}
\end{lem}

\begin{proof}
Computations similar to those in the proof of Lemma \ref{lem:g123} 
yield the required identity. Alternatively, apply the index permutation $\sigma$ 
to the equality in Lemma \ref{lem:g123}
\end{proof}

\begin{prop}\label{prop:g123om}
For any linear chord diagram $C$ with four chords, 
the element
$$
[x_1\wedge x_2\wedge x_3, C_1\tilde{\otimes}C_2\tilde{\otimes}C_3\tilde{\otimes}C_4]
$$
can be expressed as a linear combination of tensors of the form
$$
\sigma (x_1\wedge x_2\wedge x_3\otimes\omega_0\otimes\omega_0\otimes\omega_0),
$$
where $\sigma$ permutes the nine tensor factors.
In fact, we have the following equality.
\begin{align*}
&\hspace{10mm}[x_1\wedge x_2\wedge x_3, C_1\tilde{\otimes}C_2\tilde{\otimes}C_3\tilde{\otimes}C_4]=\\
&-x_1\wedge x_2\wedge x_3^{(1,j_1,9)}
\tilde{\otimes}\{C_2\tilde{\otimes}C_3\tilde{\otimes}C_4\}^{(\mathrm comp)}\\
&-x_1\wedge x_2\wedge x_3^{(1,2,j_1+1)}
\tilde{\otimes}\{C_2\tilde{\otimes}C_3\tilde{\otimes}C_4\}^{(\mathrm comp)}
+x_1\wedge x_2\wedge x_3^{(1,j_1,j_1+1)}
\tilde{\otimes}\{C_2\tilde{\otimes}C_3\tilde{\otimes}C_4\}^{(\mathrm comp)}\\
&+x_1\wedge x_2\wedge x_3^{(i_2,j_2,j_2+1)}
\tilde{\otimes}\{C_1\tilde{\otimes}C_3\tilde{\otimes}C_4\}^{(\mathrm comp)}
-x_1\wedge x_2\wedge x_3^{(i_2,i_2+1,j_2+1)}
\tilde{\otimes}\{C_1\tilde{\otimes}C_3\tilde{\otimes}C_4\}^{(\mathrm comp)}\\
&+x_1\wedge x_2\wedge x_3^{(i_3,j_3,j_3+1)}
\tilde{\otimes}\{C_1\tilde{\otimes}C_2\tilde{\otimes}C_4\}^{(\mathrm comp)}
-x_1\wedge x_2\wedge x_3^{(i_3,i_3+1,j_3+1)}
\tilde{\otimes}\{C_1\tilde{\otimes}C_2\tilde{\otimes}C_4\}^{(\mathrm comp)}\\
&+x_1\wedge x_2\wedge x_3^{(i_4,j_4,j_4+1)}
\tilde{\otimes}\{C_1\tilde{\otimes}C_2\tilde{\otimes}C_3\}^{(\mathrm comp)}
-x_1\wedge x_2\wedge x_3^{(i_4,i_4+1,j_4+1)}
\tilde{\otimes}\{C_1\tilde{\otimes}C_2\tilde{\otimes}C_3\}^{(\mathrm comp)}.\\
\end{align*}
Here $x_1\wedge x_2\wedge x_3^{(i,j,k)}\ (1\leq i<j<k\leq 9)$ means that 
we place the three tensor factors of $x_1\wedge x_2\wedge x_3$ in positions $i,j,k$ of $H^{\otimes9}$, 
and 
$\{C_l\tilde{\otimes}C_m\tilde{\otimes}C_n\}^{(\mathrm comp)} \ (1\leq l<m<n\leq 4)$ 
denotes the tensor obtained by placing 
$\widetilde C_l \tilde{\otimes} C_m \tilde{\otimes} C_n$ in the six positions complementary to $i,j,k$, 
that is, in the positions $\{1,\ldots,9\}\setminus\{i,j,k\}$.
\end{prop}

\begin{proof}
Recall again that the element $x_1\wedge x_2\wedge x_3$ is expressed as 
\begin{align*}
x_1\wedge x_2\wedge x_3&=x_1\otimes x_2\otimes x_3+x_2\otimes x_3\otimes x_1+x_3\otimes x_1\otimes x_2\\
&\quad-x_2\otimes x_1\otimes x_3-x_3\otimes x_2\otimes x_1-x_1\otimes x_3\otimes x_2
\end{align*}
consisting of $6$ components. 
Summing the first three terms on the right-hand sides of 
Lemmas \ref{lem:g123} and \ref{lem:gs123} over the six terms in the above expansion, 
we obtain the first three terms in the stated formula.
Similarly, though a little bit more complicated,  
summing the remaining six terms of the right-hand side of
Lemmas \ref{lem:g123} and \ref{lem:gs123} over the six terms in the above expansion, 
we obtain exactly the remaining six terms of the expression of the claim. 
This completes the proof.
\end{proof}

\begin{prop}\label{prop:g12yom}
For any linear chord diagram $C$ of $4$-chords, 
the element
$$
[x_1\wedge x_2\wedge y_2, C_1\tilde{\otimes}C_2\tilde{\otimes}C_3\tilde{\otimes}C_4]
$$
can be expressed as a linear combination of the form
$$
\sigma (x_1\wedge x_2\wedge y_2\otimes\omega_0\otimes\omega_0\otimes\omega_0)
$$
with various permutations $\sigma$ of $9$-words which act naturally on $H^{\otimes 9}$. 
In fact, we have the following equality.
\begin{align*}
&\hspace{10mm}[x_1\wedge x_2\wedge y_2, C_1\tilde{\otimes}C_2\tilde{\otimes}C_3\tilde{\otimes}C_4]=\\
&-x_1\wedge x_2\wedge y_2^{(1,j_1,9)}
\tilde{\otimes}\{C_2\tilde{\otimes}C_3\tilde{\otimes}C_4\}^{(\mathrm comp)}\\
&-x_1\wedge x_2\wedge y_2^{(1,2,j_1+1)}
\tilde{\otimes}\{C_2\tilde{\otimes}C_3\tilde{\otimes}C_4\}^{(\mathrm comp)}
+x_1\wedge x_2\wedge y_2^{(1,j_1,j_1+1)}
\tilde{\otimes}\{C_2\tilde{\otimes}C_3\tilde{\otimes}C_4\}^{(\mathrm comp)}\\
&+x_1\wedge x_2\wedge y_2^{(i_2,j_2,j_2+1)}
\tilde{\otimes}\{C_1\tilde{\otimes}C_3\tilde{\otimes}C_4\}^{(\mathrm comp)}
-x_1\wedge x_2\wedge y_2^{(i_2,i_2+1,j_2+1)}
\tilde{\otimes}\{C_1\tilde{\otimes}C_3\tilde{\otimes}C_4\}^{(\mathrm comp)}\\
&+x_1\wedge x_2\wedge y_2^{(i_3,j_3,j_3+1)}
\tilde{\otimes}\{C_1\tilde{\otimes}C_2\tilde{\otimes}C_4\}^{(\mathrm comp)}
-x_1\wedge x_2\wedge y_2^{(i_3,i_3+1,j_3+1)}
\tilde{\otimes}\{C_1\tilde{\otimes}C_2\tilde{\otimes}C_4\}^{(\mathrm comp)}\\
&+x_1\wedge x_2\wedge y_2^{(i_4,j_4,j_4+1)}
\tilde{\otimes}\{C_1\tilde{\otimes}C_2\tilde{\otimes}C_3\}^{(\mathrm comp)}
-x_1\wedge x_2\wedge y_2^{(i_4,i_4+1,j_4+1)}
\tilde{\otimes}\{C_1\tilde{\otimes}C_2\tilde{\otimes}C_3\}^{(\mathrm comp)}.\\
\end{align*}
Here $x_1\wedge x_2\wedge y_2^{(i,j,k)}\ (1\leq i<j<k\leq 9)$ is defined similarly to 
the one in Proposition $\ref{prop:g123om}$.
\end{prop}

\begin{proof}
Computations similar to those in the proof of Proposition \ref{prop:g123om}, 
using the corresponding versions of Lemmas \ref{lem:g123} and \ref{lem:gs123}, 
yield the stated formula.
\end{proof}

\begin{remark}
The conclusion of Proposition \ref{prop:g12yom} remains valid
if $x_1\wedge x_2\wedge y_2$ is replaced by $x_1\wedge\omega_0$, 
which is a highest weight vector of $[1]\subset \mathfrak{h}_{g,1}(1)$. 
\label{remark:xom}
\end{remark}

\begin{lem}
$\mathrm{(i)}\ $
Let 
$
C_{34}: H^{\otimes 3}\otimes H^{\otimes 2}=H^{\otimes 5}\rightarrow H^{\otimes 3}
$
be the homomorphism defined by 
$$
C_{34}(u_1\otimes u_2\otimes u_3\otimes u_4\otimes u_5)=
(u_3\cdot u_4)\, u_1\otimes u_2\otimes  u_5.
$$
Then 
\begin{align*}
C_{34}\left( (x_1\wedge x_2\wedge x_3)\otimes\omega_0\right)&=-x_1\wedge x_2\wedge x_3,\\
C_{34}\left( (x_1\wedge x_2\wedge y_2)\otimes\omega_0\right)&=-x_1\wedge x_2\wedge y_2,\\
C_{34}\left( (x_1\wedge \omega_0)\otimes\omega_0\right)&=-x_1\wedge \omega_0.
\end{align*}

$\mathrm{(ii)}\ $
Let 
$
C_{23}: H^{\otimes 2}\otimes H^{\otimes 2}=H^{\otimes 4}\rightarrow H^{\otimes 2}
$
be the homomorphism defined by 
$$
C_{23}(u_1\otimes u_2\otimes u_3\otimes u_4)=
(u_2\cdot u_3)\, u_1\otimes u_4.
$$
Then $C_{23}\left( \omega_0\otimes\omega_0\right)=-\omega_0$.
\label{lem:C234}
\end{lem}

\begin{proof}
Let $C_{12}: H\otimes H^{\otimes 2}=H^{\otimes 3}\rightarrow H$
be the homomorphism defined by 
$$
C_{12}(u_1\otimes u_2\otimes u_3)=
(u_1\cdot u_2)\, u_3.
$$
Then it is easy to see that the equalities
$$
C_{12}(x_i\otimes \omega_0)=-x_i,\quad C_{12}(y_i\otimes \omega_0)=-y_i
$$
hold for any $i$. This is because
\begin{align*}
C_{12}(x_i\otimes \omega_0)&=C_{12}\left(x_i\otimes (x_1\otimes y_1-y_1\otimes x_1
+\cdots +x_g\otimes y_g-y_g\otimes x_g)\right)=-x_i\\
C_{12}(y_i\otimes \omega_0)&=C_{12}\left(y_i\otimes (x_1\otimes y_1-y_1\otimes x_1
+\cdots +x_g\otimes y_g-y_g\otimes x_g)\right)=-y_i.
\end{align*}
The stated equalities follow immediately.
\end{proof}

\noindent
\begin{proof}[Proof of Proposition $\ref{prop:polynomial}$] 
Every element 
$$
v\in\mathfrak{h}_{g,1}(6)^{\mathrm{Sp}}\subset H^{\otimes 8}
$$ 
can be expressed as a linear combination of permutations of the tensor product 
$\omega_0\otimes\omega_0\otimes\omega_0\otimes\omega_0$ 
corresponding to various linear chord diagrams with four chords.
Moreover, as already mentioned, there is a direct sum decomposition
$\mathfrak{h}_{g,1}(1)\cong [1]\oplus [1^3]$, 
so that the bracket $[\mathfrak{h}_{g,1}(1),\mathfrak{h}_{g,1}(6)^{\mathrm{Sp}}]$
can also be decomposed into copies of 
$[1]$ and $[1^3]$.
Each of our homomorphisms $f_i, w_i, r$ and $s$ is an $\mathrm{Sp}$-equivariant projection from
$[\mathfrak{h}_{g,1}(1),\mathfrak{h}_{g,1}(6)^{\mathrm{Sp}}]$ 
to a specific summand isomorphic to $[1]$ or $[1^3]$.

Now by Proposition \ref{prop:g123om}, for the case $[1^3]$, and
Proposition \ref{prop:g12yom} (see also Remarks \ref{remark:1} and \ref{remark:xom}), 
for the case $[1]$, each of the following bracket operations 
$$
[x_1\wedge x_2\wedge x_3, C_1\tilde{\otimes}C_2\tilde{\otimes}C_3\tilde{\otimes}C_4]\
\text{and}\
[x_1\wedge x_2\wedge y_2, C_1\tilde{\otimes}C_2\tilde{\otimes}C_3\tilde{\otimes}C_4]
$$
is a linear combination of tensors of one of the following forms: 
$$
\sigma (x_1\wedge x_2\wedge x_3\otimes\omega_0\otimes\omega_0\otimes\omega_0)\
\text{and}\
\sigma (x_1\wedge x_2\wedge y_2\otimes\omega_0\otimes\omega_0\otimes\omega_0)
$$
with various permutations $\sigma$ of the nine tensor factors. 
The point here is that the result does not depend on the genus $g \geq 3$.
Our homomorphisms are defined in terms of various multi-contractions of these tensors. 
Their values are determined by the contraction formulae in Lemma \ref{lem:C234}.
Finally, any threefold multi-contraction of the tensor product
$\omega_0\otimes \omega_0\otimes \omega_0$ is a monomial in $g$ 
of degree at most $3$ (see \cite{morita99}).
The claim of polynomiality follows from this. 
\end{proof}

\section{Description of the first Galois obstruction}\label{sec:firstGal}

In this section, we prove our main results, Theorems \ref{th:main} and \ref{th:es}.

Our task is to find an element in $\mathfrak{h}_{g,1}(6)^{\mathrm{Sp}}$ that 
normalizes the Johnson image, namely an element $v$ satisfying 
$[\mathfrak{h}_{g,1}(1), v] \subset \Im \tau_{g,1}(7)$. 

\begin{prop}
For any element $v\in \mathfrak{h}_{g,1}(6)^{\mathrm{Sp}}$,
the following hold.

\noindent
$\mathrm{(i)}\ $ $[\mathfrak{h}_{g,1}(1),v]\subset \Ker ES_{7}$ if and only if $f_1(v)=0$.

\noindent
$\mathrm{(ii)}\ $  Assuming condition $\mathrm{(i)}$, 
$[\mathfrak{h}_{g,1}(1), v]\subset \mathrm{Im}\tau_{g,1}(7)$ if and only if $r(v)=0$ and $s(v)=0$.
\label{prop:n}
\end{prop}

\begin{proof} 
$\mathrm{(i)}$ As mentioned in Section \ref{sec:1_13}, the Enomoto-Satoh
trace map $ES_7$ is detected by the five homomorphisms $W_1,W_2,F_1,F_2$ and $F_3$.
Therefore, by Definition \ref{def:homo}, $v$ satisfies the required condition if
and only if all five homomorphisms $w_1,w_2,f_1,f_2$ and $f_3$ vanish on it.
By Proposition \ref{prop:poly}, the homomorphisms $w_1, f_2$ and $f_3$ vanish 
identically and
$(2g-1)f_1=-2 (2g+1)w_2$ (see Tables \ref{tab:rw} and \ref{tab:sf}).
The claim follows. 

$\mathrm{(ii)}$ This assertion follows from the second exact sequence in 
Proposition \ref{prop:str1} and the definitions of the homomorphisms $r$ and $s$. 

Here we remark that, in the case $g=2$, only the homomorphisms $f_1$ and $r$ are defined 
while $w_2$ and $s$ are not. However, the above argument is also valid 
in this case with a minor modification.
\end{proof}

Now we consider the extra element
\begin{align*}
\varepsilon&=\frac{2v_1}{2g(2g+1)(2g+2)(2g+3)}\\
&+\frac{8v_3}{(2g-2)2g(2g+1)(2g+2)}+\frac{3v_4}{(2g-2)(2g-1)2g(2g+1)}\in \Ker  ES_6
\end{align*}
defined in Section \ref{in}. It is known that the kernel of the Enomoto-Satoh trace map is a Lie 
subalgebra of $\mathfrak{h}_{g,1}$, and hence $ES_7([\mathfrak{h}_{g,1}(1),\varepsilon])=0$.
It follows that 
$$
f_1(\varepsilon)=0.
$$
We can check this directly by the polynomial of $f_1$ at $\varepsilon$ given in Table \ref{tab:sf}.
Thus $\varepsilon$ satisfies the condition $\mathrm{(i)}$ of Proposition \ref{prop:n}.

Next we compute $r(\varepsilon)$.

\begin{prop}
The value of the homomorphism $r$ on the extra element
is given by
$$
r(\varepsilon)=3 \frac{2g+1}{(g-1)g}
$$
\label{prop:re}
\end{prop}

\begin{proof}
By Table \ref{tab:rw}, we see that 
\begin{align*}
r(v_1)&=-24(g+1)(2g+1)(2g+3),\ r(v_3)=-36(g+1)(2g+1), \\
r(v_4)&=12(2g-1)(2g+1)(2g+3), 
\end{align*}
so that
\begin{align*}
&r(2g(2g+1)(2g-2)(2g-1)(2g+2)(2g+3)\varepsilon)\\
&=r(2(2g-2)(2g-1)v_1+8(2g-1)(2g+3)v_3+3(2g+2)(2g+3)v_4)\\
&=2(2g-2)(2g-1)\cdot (-24)(g+1)(2g+1)(2g+3)\\
&\quad +8(2g-1)(2g+3)\cdot (-36)(g+1)(2g+1)\\
&\quad +3(2g+2)(2g+3)\cdot 12(2g-1)(2g+1)(2g+3)\\
&=(2g-1)(2g+1)(2g+3)(g+1)\{-2\cdot 24 (2g-2)-8\cdot 36+6\cdot 12 (2g+3)\}\\
&=24 (2g-1)(2g+1)^2(2g+3)(g+1).
\end{align*}
Hence
$$
r(\varepsilon)=3 \frac{2g+1}{(g-1)g}.
$$
This completes the proof. 
\end{proof}

Thus we have confirmed that the extra element $\varepsilon$ does not normalize the
Johnson image while its projection to $\mathfrak{h}_{g,*}(6)$ does.
Our next task is to find a ``correction term" to add to $\varepsilon$ so that the resulting element  
normalizes the Johnson image. 
In view of Proposition \ref{prop:n}, we seek such an element
which should be contained in $\mathfrak{j}_{g,1}(6)$.

\begin{prop}
$f_1(2j_1-3j_2)=0$, and hence $[\mathfrak{h}_{g,1}(1),2j_1-3j_2]$ is contained in $\Ker  ES_{7}$.
\end{prop}

\begin{proof}
Recall that
$j_1=v_2-5v_3-4v_4+2v_5,
j_2=3v_1+3v_2+3v_3-v_4-2v_5$. It follows from Table \ref{tab:sf} that 
\begin{align*}
f_1(j_1)&=(g+1)(2g-1)(8(g+13)-5\cdot 24g-4\cdot (-16)(2g-1)+2\cdot (-8)(g-2))\\
&=72(g+1)(2g-1)\\
f_1(j_2)&=(g+1)(2g-1)\\
&\quad \cdot (3(-24)(2g+3)+3\cdot 8(g+13)+3\cdot 24g-(-16)(2g-1)-2\cdot (-8)(g-2))\\
&=48(g+1)(2g-1).
\end{align*}
Therefore $f_1(2j_1-3j_2)=0$. 
The second assertion follows from Proposition \ref{prop:n}.
\end{proof}

Next we compute $r(2j_1-3j_2)$.
\begin{prop} The equality 
$r(2j_1-3j_2)=2^2\cdot 3^3\cdot 5 \cdot (2g+1)$ holds.
\label{prop:rj}
\end{prop}

\begin{proof}
We have
\[2j_1-3j_2=-9v_1-7v_2-19v_3-5v_4+10v_5.\]
Then the result follows by substituting the values of $r(v_i)$ given in Table \ref{tab:rw}. 
Alternatively, the result follows from the explicit values 
in Table \ref {tab:rjj} together with Corollary \ref{cor:4g}.
\begin{table}[h]
\caption{$\text{The value of the homomorphism $r$ on $2j_1-3j_2$}$}
\begin{center}
\begin{tabular}{|c|c|}
\noalign{\hrule height0.8pt}
\hfil $g$ & $r(2j_1-3j_2)$  \\
\hline
$2$ & $2700=2^2\cdot 3^3\cdot 5\cdot 5$  \\
\hline
$3$ & $3780=2^2\cdot 3^3\cdot 5\cdot 7$  \\
\hline
$4$ & $4860=2^2\cdot 3^5\cdot 5=2^2\cdot 3^3\cdot 5 \cdot 9$\\
\hline
$5$ & $5940=2^2\cdot 3^3\cdot{5}\cdot {11}$ \\
\hline
$6$ & $7020=2^2\cdot 3^3\cdot 5\cdot 13$ \\
\hline
$7$ & $8100=2^2\cdot 3^4\cdot {5}^2=2^2\cdot 3^3\cdot 5\cdot 15$ \\                                                               
\noalign{\hrule height0.8pt}
\end{tabular}
\end{center}
\label{tab:rjj}
\end{table}
\end{proof}

In view of Propositions \ref{prop:re} and \ref{prop:rj},
if we define
$$
j=2j_1-3j_2
$$
then we have
\[r(180 (g-1)g \ \varepsilon-j)=540 (2g+1)-2^2\cdot 3^3\cdot 5 \cdot (2g+1)=0.\]

We now prove that $\sigma_3(g)=180 (g-1)g\ \varepsilon-j$, that is, $180 (g-1)g\ \varepsilon-j$ 
represents the first Galois obstruction.
We already know that $f_1(180 (g-1)g\ \varepsilon-j)=0$ and 
$r(180 (g-1)g\ \varepsilon-j)=0$. Therefore, in view of Proposition \ref{prop:n}, it remains to show that
\[s(180 (g-1)g\ \varepsilon-j)=0.\]

\begin{prop}\label{prof:esep}
The value of the homomorphism $s$ on the extra element $\varepsilon$ is given by
$$
s(\varepsilon)=3 \frac{g+1}{(g-1)g}.
$$
\label{prop:se}
\end{prop}

\begin{proof}
By Table \ref{tab:sf}, we see that 
$$
s(v_1)=0,\quad  s(v_3)=-12(g+1)^2(2g+3), \quad 
s(v_4)=4(g+1)(2g-1)(10g+13).
$$
We have
\begin{align*}
&s(2g(2g+1)(2g-2)(2g-1)(2g+2)(2g+3)\varepsilon)\\
=&s(2(2g-2)(2g-1)v_1+8(2g-1)(2g+3)v_3\\
&\hspace{2cm}+3(2g+2)(2g+3)v_4)\\
=&8(2g-1)(2g+3)\cdot (-12)(g+1)^2(2g+3)\\
+&3(2g+2)(2g+3)\cdot 4(g+1)(2g-1)(10g+13)\\
=&(2g-1)(2g+3)(g+1)\{8\cdot (-12)(g+1)(2g+3)+6\cdot 4(g+1)(10g+13)\}\\
=&24 (2g-1)(2g+3)(g+1)^2\{-4(2g+3)+10g+13\}\\
=&24 (2g-1)(2g+1)(2g+3)(g+1)^2.
\end{align*}
Hence
$$
s(\varepsilon)=3 \frac{g+1}{(g-1)g}.
$$
This completes the proof. 
\end{proof}

Next we compute $s(2j_1-3j_2)$.
\begin{prop} The equality $s(2j_1-3j_2)=2^2\cdot 3^3\cdot 5 \cdot (g+1)$ holds.
\label{prop:sj}
\end{prop}

\begin{proof}
We have
$$
2j_1-3j_2=-9v_1-7v_2-19v_3-5v_4+10v_5.
$$
Then the result follows by substituting the values of $s(v_i)$ given in Table \ref{tab:sf}. 
Alternatively, the result follows from the explicit values in Table \ref {tab:sjj} 
together with Corollary \ref{cor:4g}.
\begin{table}[h]
\caption{$\text{The value of the homomorphism $s$ on $2j_1-3j_2$}$}
\begin{center}
\begin{tabular}{|c|c|}
\noalign{\hrule height0.8pt}
\hfil $g$ & $s(2j_1-3j_2)$  \\
\hline
$3$ & $2160=2^4\cdot 3^3\cdot 5=2^2\cdot 3^3\cdot 5\cdot 4$  \\
\hline
$4$ & $2700=2^2\cdot 3^3\cdot 5^2=2^2\cdot 3^3\cdot 5 \cdot 5$\\
\hline
$5$ & $3240=2^3\cdot 3^4\cdot{5}=2^2\cdot 3^3\cdot 5\cdot 6$ \\
\hline
$6$ & $3780=2^2\cdot 3^3\cdot 5\cdot 7$ \\
\hline
$7$ & $4320=2^5\cdot 3^3\cdot {5}=2^2\cdot 3^3\cdot 5\cdot 8$ \\                                                               
\noalign{\hrule height0.8pt}
\end{tabular}
\end{center}
\label{tab:sjj}
\end{table}
\end{proof}

\noindent
\begin{proof}[Proof of Theorem $\ref{th:main}$]
By Propositions \ref{prof:esep} and \ref{prop:sj}, we have
\[s(180 (g-1)g \ \varepsilon-j)=540(g+1)-540 (g+1)=0.\]
Hence $180 (g-1)g\ \varepsilon-j$ is a 
generator of
$\mathcal{N}(6)/\mathrm{Im}\tau_{g,1}(6)$. 
Note that such an element is unique up to multiplication by a nonzero scalar 
and addition of an element of the Johnson image.
This completes the proof.
\end{proof}

\vspace{3mm}
\begin{proof}[Proof of Theorem $\ref{th:es}$]
We have to prove the equality 
\[ES_6^1(\sigma_3(g))/ES_6^2(\sigma_3(g))=2g-1.\]
Since $ES_6(\varepsilon)=0$, and also we are considering the ratio,
we have
\[ES_6^1(\sigma_3(g))/ES_6^2(\sigma_3(g))=ES_6^1(j)/ES_6^2(j).\]
We compute $ES_6^i(j)$ for $i=1,2$ using the values of the
Enomoto-Satoh trace maps on the basis elements $v_i\ (i=1, 2, 3, 4, 5)$. 
These values are obtained by multiplying the values on the normalized basis elements $\bar{v_i}$ 
given in Table \ref{tab:ESnn} by the corresponding eigenvalues given in Table \ref{tab:ev}.
\begin{align*}
ES_6^1(j)&=
\begin{pmatrix}
-9 & -7 & -19 & -5 & 10
\end{pmatrix}
\begin{pmatrix}
6\cdot 2g(2g+1)(2g+2)(2g+3)\\
27\cdot (2g-2)2g(2g+1)(2g+2)\\
3\cdot (2g-2)2g(2g+1)(2g+2)\\
-12\cdot (2g-2)(2g-1)2g(2g+1)\\
-3\cdot (2g-4)(2g-2)2g(2g+1)
\end{pmatrix}
\\
&=2g(2g+1)\{-9\cdot 6 (2g+2)(2g+3)-7\cdot 27 (2g-2)(2g+2)\\
&\hspace{3mm}-19\cdot 3(2g-2)(2g+2)-5\cdot (-12)(2g-2)(2g-1)+10\cdot (-3)(2g-4)(2g-2)\}\\
&=2g(2g+1)\cdot \{-270(2g)^2-270 (2g)+540\}\\
&=2g(2g+1)\cdot (-270) (2g-1)(2g+2)
\end{align*}

\begin{align*}
ES_6^2(j)&=
\begin{pmatrix}
-9 & -7 & -19 & -5 & 10
\end{pmatrix}
\begin{pmatrix}
6\cdot 2g(2g+1)(2g+2)(2g+3)\\
-1\cdot (2g-2)2g(2g+1)(2g+2)\\
-3\cdot (2g-2)2g(2g+1)(2g+2)\\
4\cdot (2g-2)(2g-1)2g(2g+1)\\
1\cdot (2g-4)(2g-2)2g(2g+1)
\end{pmatrix}
\\
&=2g(2g+1)\{-9\cdot 6 (2g+2)(2g+3)-7\cdot (-1)(2g-2)(2g+2)\\
&\hspace{3mm}-3\cdot (-3)(2g-2)(2g+2)-5\cdot 4(2g-2)(2g-1)+10\cdot (2g-4)(2g-2)\}\\
&=2g(2g+1)\cdot \{-270 (2g)-540\}\\
&=2g(2g+1)\cdot (-270)(2g+2)
\end{align*}
Hence $ES_6^1(j)/ES_6^2(j)=2g-1$. This completes the proof.
\end{proof}

\section{Questions}

Let $\sigma_{2k+1}(g)\in \mathfrak{h}_{g,1}(4k+2)^{\mathrm{Sp}}\ (g=1,2,3,\ldots, k=1,2,\ldots)$ be the 
Galois obstructions. 
As mentioned in Introduction, they are not canonically defined. 
However, if we consider these elements 
modulo the commutator ideal of $\Im \hat{\tau}_{g,1}$, then
they are canonically defined up to
nonzero scalar and modulo the Johnson image. 
Now they behave considerably differently
in the cases $g=1$ and $g\ge 2$.
For a detailed study for the former case from the point of view
of number theory, we refer to the work
\cite{hma} of Hain and Matsumoto. Here we consider these
elements from the topological viewpoint.

If $g=1$, then $\mathcal{L}_{1,1}(2)=\wedge^2 H\cong \Q$ generated by 
$\omega_0=x\wedge y$. Hence, for any $k\geq 2$, $\mathcal{L}_{1,1}(k)$ 
coincides with the degree $k$ part of the Lie ideal generated by $\omega_0$. 
Together with the fact $\mathfrak{h}_{1,1}(1)=0$, we see that 
\[\mathfrak{h}_{1,1}=\mathfrak{j}_{1,1},\]
so that 
$\sigma_{2k+1}(1)\in \mathfrak{j}_{1,1}(4k+2)$ for all $k$. However, this is an
exceptional case. To see why this case is exceptional, 
we recall from \cite[Theorem 3.6]{mss4} that 
we have a direct sum decomposition
\begin{equation}\label{eq:decomp}
\mathfrak{h}_{g,1}(k)=
\mathfrak{j}_{g,1}(k)\oplus \mathcal{L}_g(k)\oplus \Im \tau_g(k)\oplus \Coker \tau_g(k).
\end{equation}
Here $\mathcal{L}_g=\oplus_{k=1}^\infty \mathcal{L}_g (k)$ 
is the quotient graded Lie algebra of $\mathcal{L}_{g,1}$ by the Lie ideal generated 
by $\omega_0$, and $\tau_g$ is the (rational) Johnson homomorphism 
associated with the Johnson filtration of the mapping class group of the closed surface $\Sigma_g$. 
The target of $\tau_g$ is given by $\mathfrak{h}_g =\oplus_{k=1}^\infty \mathfrak{h}_g (k)$ with 
\[\mathfrak{h}_{g}(k)=\mathfrak{h}_{g,1}(k)/(\mathfrak{j}_{g,1}(k)\oplus\mathcal{L}_g(k)).\]
In view of the decomposition (\ref{eq:decomp}), 
it 
follows from the works of Takao \cite{takao} and Hain \cite{hain} that 
\[\sigma_{2k+1}(g)\not\in \mathfrak{j}_{g,1}(4k+2)^{\mathrm{Sp}}\oplus \mathcal{L}_{g}(4k+2)^{\mathrm{Sp}}\]
for all $g\geq 2$ and $k \ge 1$, 
because the arithmetic Johnson homomorphism is also defined in the setting of closed surfaces, 
although Hain describes the case of compact surfaces with one boundary component.

Let 
\[ES_{4k+2} \colon \mathfrak{h}_{g,1}(4k+2)^{\mathrm{Sp}}\longrightarrow
\mathfrak{a}_{g}(4k)^{\mathrm{Sp}}\]
be the Enomoto-Satoh trace map.
Then, as already mentioned in $\S$ \ref{sec:results}, 
it is explained in Hain's paper \cite{hain} that results of Alekseev, Kawazumi, Kuno and Naef
show that the Turaev cobracket of $\sigma_{2k+1}(g)$ is non-trivial for all $g \ge 2$ and $k \ge 1$. 
Their results also imply that 
\[ES_{4k+2}(\sigma_{2k+1}(g))\not=0.\]
In this paper, we have determined the value of $ES_6(\sigma_3(g))=ES_6(-j)\ (\not=0)$ explicitly.

\begin{question}
Is it possible to compute the Turaev cobracket of the element $j\in \mathfrak{j}_{g,1}(6)^{\mathrm{Sp}}$?
\end{question}

On the other hand, 
the non-triviality of $ES_6(\sigma_3(g))$ arises in a rather subtle way from the ideal $\mathfrak j_{g,1}$. 
In fact, if we pass to the quotient by this ideal, then 
$\Ker  ES_6$ and $\mathcal{N}(6)^{\mathrm{Sp}}$ coincide.
We therefore define
$$
\overline{ES}_k\colon \mathfrak{h}_{g}(k)\longrightarrow \mathfrak{a}_g(k-2)/ES_k(\mathfrak{j}_{g,1}(k)).
$$
Note that $ES_k=0$ on $\mathcal{L}_g(k)$, so the above 
definition is well-defined. We may call this homomorphism
$\overline{ES}_k$ the {\it reduced} Enomoto-Satoh trace map.

\begin{question}

\noindent
Let 
$$
\overline{ES}_{4k+2} \colon \mathfrak{h}_{g}(4k+2)^{\mathrm{Sp}}\longrightarrow
\mathfrak{a}_{g}(4k)^{\mathrm{Sp}}/ES_{4k+2}(\mathfrak{j}_{g,1}(4k+2))
$$
be the reduced Enomoto-Satoh trace map. Is it true that
$$
\overline{ES}_{4k+2}(\sigma_{2k+1}(g))=0?
$$
We have shown that this holds for $k=1$ and $g\ge 2$.

If the answer to the above question is affirmative, 
what is the meaning of the ``correction term",
which should be an element of $\mathfrak{j}_{g,1}(4k+2)^{\mathrm{Sp}}$?

\label{probles}
\end{question}

Hain \cite{hain} proved that the Turaev cobracket and hence the Enomoto-Satoh
trace map vanish on the arithmetic Johnson image of the {\it commutator ideal}
of the motivic Lie algebra $\mathcal{L}(\sigma_3,\sigma_5,\ldots)$.
If the answer to the above question is affirmative, then the {\it reduced} Enomoto-Satoh trace
map vanishes on the entire arithmetic Johnson image.
Conant \cite{conant}, Conant-Kassabov \cite{ck} and Kuno-Sato \cite{ks} 
constructed further obstructions to the Johnson image beyond the Enomoto-Satoh trace map.
However, it seems that no obstruction is known at present which detects any of 
the Galois obstructions in the Johnson cokernel of the mapping class group of a {\it closed} surface.
The following is a deep question concerning the structure of the Johnson cokernel. 

\begin{question}
\noindent
Are there effective methods for distinguishing the arithmetic Johnson image
from the original geometric Johnson image
at the level of the mapping class group of a closed surface?
\end{question}

Another related deep question is whether there are effective methods for determining the abelianization 
$H_1(\mathfrak{h}_{g,1})$ of the Lie algebra $\mathfrak{h}_{g,1}$, in light of the works of 
Conant-Kassabov-Vogtmann \cite{ckv} and Bartholdi \cite{b} (see also \cite{mssa}).
We note that a fundamental result of Hain \cite{haint} states that 
the Johnson image is equal to the Lie subalgebra of $\mathfrak{h}_{g,1}$ generated by the 
degree $1$ 
Johnson image $\Im \tau_{g,1} (1)=\mathfrak{h}_{g,1}(1)=\wedge^3 H$ in degree $1$. 
Consequently, every element of the abelianization $H_1(\mathfrak{h}_{g,1})$ outside 
$\wedge^3 H$ must come from the Johnson cokernel.

\begin{question}
\noindent
For any fixed $k$, what is the relation among $\sigma_{2k+1}(g)$ for $g=1,2,3,\ldots$?
\end{question}

For $k=1$, our result shows that the Galois obstruction can be 
written as the sum of two terms: $\sigma_3(g)=180 (g-1)g\ \varepsilon-j\ (g\geq 2)$.
These two terms are considerably different in character. 
For example, the essential term 
$\varepsilon\in \Ker  ES_6$ is expressed as a linear combination of the normalized basis elements, 
while the ``correction term" $j\in \mathfrak{j}_{g,1}(6)^{\mathrm{Sp}}$ is expressed as 
a linear combination of the usual basis elements. In both cases, the coefficients are 
independent of $g$ up to an overall scalar factor.
In the terminology of \cite{mss4}, the former term is {\it i-stable} while the latter term
is {\it p-stable}.

\bibliographystyle{amsplain}

\begin{thebibliography}{30}

\bibitem{akkn}
A. ~Alekseev, N. ~Kawazumi, Y. ~Kuno, F. ~Naef,
\textit{The Goldman-Turaev Lie bialgebra and the Kashiwara-Vergne problem in higher genera}, 
arXiv:1804.09566 [math.GT].

\bibitem{an}
M.~Asada, H.~Nakamura, 
\textit{On the graded quotient modules of mapping class groups of surfaces},
Israel J.\ Math. 90 (1995), 93--113. 

\bibitem{b}
L. ~Bartholdi, 
\textit{The rational homology of the outer automorphism group of $F_7$}, 
New York J. Math. 22 (2016), 191--197.

\bibitem{brown}
F.~Brown, 
\textit{Mixed Tate motives over $\mathbb{Z}$}, Ann.\ Math. 175 (2012), 949--976. 

\bibitem{conant}
J. ~Conant, 
\textit{The Johnson cokernel and the Enomoto-Satoh invariant}, 
Algebr.\ Geom.\ Topol.\ 15 (2015), 801--821.


\bibitem{ck}
J. ~Conant, M. ~Kassabov,
\textit{Hopf algebras and invariants of the Johnson cokernel}, 
Algebr.\ Geom.\ Topol.\ 16 (2016), 2325--2363.


\bibitem{ckv}
J. ~Conant, M. ~Kassabov, K. ~Vogtmann, 
\textit{Hairy graphs and the unstable homology of
$\mathrm{Mod}(g,s)$, $\mathrm{Out}(F_n)$ and $\mathrm{Aut}(F_n)$}, 
J. \ Topol. 6 (2013), 119--153.

\bibitem{es}
N.~Enomoto, T.~Satoh, 
\textit{New series in the Johnson cokernels of the mapping class groups of surfaces}, 
Algebr.\ Geom.\ Topol.\ 14 (2014), 627--669.

\bibitem{FNW}
M.~Felder, F.~Naef, T.~Willwacher, 
\textit{Stable cohomology of graph complexes}, 
Selecta Math. 29 (2023), no. 2, Paper No. 23, 72 pp.

\bibitem{fh}
W.~Fulton, J.~Harris, \textit{Representation Theory}, 
Graduate Texts in Mathematics 129, Springer-Verlag, 1991.

\bibitem{gl} S.~Garoufalidis, J.~Levine, 
\textit{Tree-level invariants of three-manifolds, Massey products and the Johnson homomorphism}, 
Graphs and patterns in mathematics and theoretical physics, 
Proc.\ Sympos.\ Pure Math. 73 (2005), 173--205.

\bibitem{hm}
K.~Habiro, G.~Massuyeau, 
\textit{From mapping class groups to monoids of homology cobordisms: a survey}, 
Handbook of Teichm\"uller theory volume III (editor: A. Papadopoulos) (2012), 465--529.


\bibitem{haint}
R.~Hain, 
\textit{Infinitesimal presentations of the Torelli groups}, 
J.\ Amer. \ Math. \ Soc. 10 (1997), 597--651.

\bibitem{hain}
R. Hain,
\textit{Johnson homomorphisms}, 
EMS Surveys \ Math.\ Sci. 7 (2020), 33--116.

\bibitem{hma}
R. ~Hain, M. ~Matsumoto,
\textit{Universal mixed elliptic motives}, 
J. Inst. Math. Jussieu {\bf 19} (2020), 663--766. 


\bibitem{j}
D. ~Johnson,
\textit{An abelian quotient of the mapping class group $\mathcal{I}_g$},
Math. Ann. 249 (1980), 225--242.

\bibitem{j2}
D. ~Johnson, 
\textit{A survey of the Torelli group}, 
in: Low-dimensional topology (San Francisco, Calif., 1981), 
Contemp. Math. {\bf 20}, 165--179. Amer. Math. Soc., 
Providence, RI, 1983.

\bibitem{kk}
N.~Kawazumi, Y.~Kuno,
\textit{The Goldman-Turaev Lie bialgebra and the Johnson homomorphisms},
``Handbook of Teichm\"uller theory'', edited by A. Papadopoulos,
Volume V, EMS Publishing House, Z\"urich, 2015, 98--165. 

\bibitem{ks}
Y.~Kuno, M.~Sato, 
\textit{On the 2-loop part of the Johnson cokernel}, 
arXiv:2508.19041 [math.GT].

\bibitem{KRW}
A.~Kupers, O.~Randal-Williams, 
\textit{On the Torelli Lie algebra}, 
Forum Math. Pi 11 (2023), Paper No. e13, 47 pp.

\bibitem{matsumoto}
M.~Matsumoto, 
\textit{Galois representations on profinite braid groups on curves}, 
J.\ Reine Angew.\ Math. 474 (1996), 169--219. 

\bibitem{matsumotopcmi}
M.~Matsumoto, 
\textit{Introduction to arithmetic mapping class groups}, 
in ``Moduli Spaces of Riemann Surfaces", 
IAS/Park City Mathematical Series, Vol. 20,
321--356, Amer. Math. Soc., 2013.

\bibitem{morita93} S.~Morita, 
\textit{Abelian quotients of subgroups of the mapping class group of surfaces}, 
Duke Math.\ J. 70 (1993), 699--726.

\bibitem{morita99} S.~Morita, 
\textit{Structure of the mapping class groups 
of surfaces: a survey and a prospect}, 
Geometry and Topology monographs 2, 
\textit{Proceedings of the Kirbyfest} (1999), 349--406.

\bibitem{mss4}
S. ~Morita, T. ~Sakasai, M. ~Suzuki, 
\textit{Structure of symplectic invariant Lie subalgebras of symplectic derivation Lie algebras}, 
Advances in Mathematics 282 (2015), 291--334.

\bibitem{mssa}
S. ~Morita, T. ~Sakasai, M. ~Suzuki, 
\textit{An abelian quotient of the symplectic derivation Lie algebra of the free Lie algebra}, 
Exp.\ Math. 27 (2018), 302--315.

\bibitem{NW}
F.~Naef, T.~Willwacher, 
\textit{The Johnson homomorphism, embedding calculus and graph complexes}, 
arXiv:2602.09915 [math.QA].

\bibitem{nakamura}
H.~Nakamura, 
\textit{Coupling of universal monodromy representations of Galois-Teichm\"uller 
modular groups}, Math.\  Ann. 304 (1996), 99--119.

\bibitem{takao} 
N.~Takao, \textit{Braid monodromies on proper curves and pro-$\ell$ Galois representations}, 
J.\ Inst.\ Math.\ Jussieu 11 (2012), 161--181.

\end{thebibliography}

\end{document}